\documentclass[11pt]{article}
\usepackage{amssymb}
\usepackage{amsthm}
\usepackage{amsmath,amsfonts,amssymb}
\usepackage{graphicx}
\usepackage{algpseudocode}
\usepackage{algorithm}
\usepackage{natbib}
\usepackage{color}
\usepackage{framed}

\newtheorem{theorem}{Theorem}[section]
\newtheorem{lemma}[theorem]{Lemma}
\newtheorem{proposition}[theorem]{Proposition}

\newtheorem{definition}[theorem]{Definition}
\newtheorem{remark}[theorem]{Remark}

\def\sqr#1#2{\vbox{\hrule height .#2pt
\hbox{\vrule width .#2pt height #1pt \kern #1pt
\vrule width .#2pt}\hrule height .#2pt }}

\def\begi{\begin{itemize}}
\def\endi{\end{itemize}}
\def\bega{\begin{array}}
\def\enda{\end{array}}

\def\bel{\begin{equation}\label}
\def\eeq{\end{equation}}

\begin{document}

 \begin{center}
\textcolor{blue}{ARTICLE LINK:  http://www.tandfonline.com/doi/abs/10.1080/21680566.2015.1064793?journalCode=ttrb20
\\  PLEASE CITE THIS ARTICLE AS\\ 
Han, K., Piccoli, B., Szeto, W.Y., 2015. Continuous-time link-based kinematic wave model: Formulation, solution existence and well-posedness. Transportmetrica B: Transport Dynamics, DOI: 10.1080/21680566.2015.1064793.}
 \line(1,0){469}
 \end{center}
 
\title{Continuous-time link-based kinematic wave model: Formulation, solution existence and well-posedness}
\author{Ke Han$^{a}\thanks{e-mail: k.han@imperial.ac.uk;}$
\qquad Benedetto Piccoli$^{b}\thanks{e-mail: piccoli@camden.rutgers.edu;}$
\qquad W. Y. Szeto$^{c}\thanks{Corresponding author, e-mail: ceszeto@hku.hk;}$ \\\\
$^{a}$\textit{Department of Civil and Environmental Engineering, Imperial College London, UK}\\
$^{b}$\textit{Department of Mathematics, Rutgers University, Camden, USA}\\
$^{c}$\textit{Department of  Civil Engineering, the University of Hong Kong, China.}}
\date{}
\maketitle

\begin{abstract}
We present a continuous-time {\it link-based kinematic wave} model (LKWM) for dynamic traffic networks based on the scalar conservation law model \citep{Lighthill and Whitham, Richards}. Derivation of the LKWM involves the variational principle for the Hamilton-Jacobi equation  and junction models defined via the notions of demand and supply. We show that the proposed LKWM can be formulated as a system of {\it differential algebraic equations} (DAEs), which captures shock formation and propagation, as well as queue spillback. The DAE system, as we show in this paper,  is the continuous-time counterpart of the {\it link transmission model} \citep{LTM}. In addition, we present a solution existence theory for the continuous-time network model and investigate continuous dependence of the solution on the initial data, a property known as well-posedness. We test the DAE system extensively on several small and large networks and demonstrate its numerical efficiency. 
\end{abstract}

\section{\label{Intro}Introduction}
First-order scalar conservation law models of the LWR type \citep{Lighthill and Whitham, Richards} have been widely used in the traffic science and engineering literature, due to its relatively simple mathematical structure and the capability of capturing realistic traffic network phenomena such as shock waves and vehicle spillback. The LWR model describes the temporal-spatial evolution of vehicle density on a homogeneous link via a scalar conservation law of the form 
\begin{equation}\label{LWRPDE}
\partial_t\rho(t,\,x)+  \partial_x  f\big(\rho(t,\,x)\big)~=~0
\end{equation}

\noindent where the link is expressed as a spatial interval $[a,\,b]$;  $\rho(t,\,x)$ denotes vehicle density; $f(\cdot): [0,\,\rho^{jam}]\rightarrow [0,\,C]$ expresses vehicle flow as a function of time, and is called the fundamental diagram; $\rho^{jam}$ denotes the jam density; $C$ is the link flow capacity. 

Classical mathematical results on the first-order hyperbolic equations of the form \eqref{LWRPDE} can be found in \cite{Bbook}. For a detailed discussion of numerical schemes of conservation laws, we refer the readers to \cite{Godunov} and \cite{LeVeque}. The {\it cell transmission model} (CTM), which is a discrete version of the LWR model assuming trapezoidal fundamental diagram, is proposed by \cite{CTM1} and extended to traffic networks by \cite{CTM2}. Another discrete-time version of the LWR model, derived from the variational principle \citep{VT1, Newella} with a triangular fundamental diagram, is called the {\it link transmission model} \citep{LTM}. The LWR model is also related to the point queue models. \cite{GVM1} and \cite{GVM2} discuss the connection between the LWR model and the Vickrey model and propose a computation method based on the variational method.

Network extensions of the LWR model, or any of its discrete-time versions, require treatment of road junction models, which are typically articulated via the notion of {\it Riemann Solvers} (RS). Work in this regard includes but is not limited to \cite{Bretti et al, Chitour and Piccoli, CGP, D'Apice and Piccoli, GP2009,  Herty and Klar, HR1995, Jin 2010,  LK1999, LK2002} and \cite{Ni and Leonard 2005}. The Riemann Solver may take various forms, depending on specific assumptions made at the road intersection under consideration. Some common features of existing Riemann Solvers include the maximization of flow through the junction \citep{CTM2, CGP, GPbook}, the prescribed vehicle turning ratio \citep{HR1995, Jin 2010}, and the driving priorities \citep{CGP, GP2009}.

Notably, for the past few decades the LWR model, along with its variations, have evolved from ones developed to describe the relationship among aggregated quantities of flow, density and velocity \citep{Sumalee, Szeto2008}, to the generation of models that are network-based; provide critical information on speed, acceleration, location and concentration to allow accurate estimation of emissions \citep{ZYC, Zhu2013}, and treat spillback and gridlock phenomena \citep{spillback, HGPFY, LS2002, Szeto 2003, Szeto and Lo 2006}. Moreover, the LWR-based traffic models play a key role in dynamic traffic assignment models \citep{Balijepalli, FHNMY, MU, Szeto2011}. Despite their popularity in engineering applications, such a class of fluid models are mainly presented and computed in discrete time; and their qualitative properties pertaining to existence, uniqueness, and well-posedness \footnote{In the mathematical modeling of a physical system, the term {\it well-posedness} refers to the property of having a unique solution, and the behavior of that solution hardly changes when there is a slight change in the initial/boundary conditions.} are relatively less understood, especially when one visits the continuous-time domain of these models. This paper is among the first to bridge this gap by proposing a continuous-time formulation of the network-based kinematic wave model with simplified assumption on the fundamental diagram, and by analyzing its qualitative properties pertaining to solution existence and uniqueness, and well-posedness.

\subsection{The link-based kinematic wave model and the DAE system}

This paper proposes a continuous-time hydrodynamic model called the {\it link-based kinematic wave model} (LKWM) for networks based on triangular fundamental diagrams. We show that such a network model can be derived based on the variational theory studied in a number of papers \citep{ABP, CC1, CC2, VT1, Newella, Newellb, Newellc}. We approach such a continuous-time representation of the simplified kinematic wave model by invoking the concept of {\it separating shock} \citep{Bretti et al}, which divides the link into the congested region (corresponding to the left half of the fundamental diagram) and the uncongested region (corresponding to the right half of the fundamental diagram). It can be shown that such a separating shock is unique under some minor conditions (e.g. the link is initially empty) \citep{Bretti et al}. We show that although the variational principle cannot directly identify the time-varying location of the separating shock in the interior of the link, it can capture two special circumstances where the separating shock reaches the entrance or exit of the link. When the separating shock reaches either boundary of the link (in this case we say that the separating shock is latent), one needs to revise the link demand or supply, which otherwise stays constant and is equal to the flow capacity according to \cite{LK1999}. Therefore, in order to determine link demand and supply, it suffices to invoke the variational theory to detect the presence of latent separating shock.

We employ a set of binary variables to indicate the congested/uncongested state of the link entrance and exit.  Two {\it Riemann Solvers} \citep{GPbook} are introduced to describe, respectively, a merge and a diverge junction model. The Riemann Solvers,  when combined with the LKWM, leads to a DAE system representation of the dynamic network model of interest. The proposed DAE system is continuous-time in nature, and is capable of capturing some realistic traffic phenomena such as shock waves and queue spillback.  Moreover, the idea of developing a DAE system for the LWR-based network model is non-trivial as the LWR model is based on partial differential equations instead of ordinary differential equations. A more straightforward treatment of the continuous-time LWR-based network models is through {\it partial differential algebraic equations} (PDAE) system;  see \cite{Kachani and Perakis} and \cite{Perakis and Roels} for example.

The proposed DAE system eliminates  partial derivatives (i.e., derivative with respect to space) since it is derived from a variational theory for the Hamilton-Jacobi equation representation of the link dynamics. In addition, the proposed DAE system treats vehicle flows and cumulative curves as its primary variables; this is in contrast to some existing models that describes the evolution of vehicle densities and/or speeds (e.g., the LWR model and the CTM). This feature facilitates applications of the LKWM to traffic control problems where the control directly acts on vehicle flows, such as ramp metering and signal control. Moreover, the binary variables, when combined with a time-discretization of the LKWM, naturally lead to a set of linear constraints, which can be applied to network optimization problems formulated as mathematical programs; examples include traffic signal optimization problems  \citep{HGPFY, LHGFY} and system optimal dynamic traffic assignment problems \citep{Z}.

More importantly, the proposed continuous-time representation of the LWR model on networks is valuable for the modeling of traffic flow on networks for the following reasons. Unlike all discrete models, the definition and behavior of a continuous-time model is independent of time step size chosen or any numerical apparatus employed. It serves as the benchmark to all discrete approximations, such as the cell transmission model or the link transmission model, by predicting and dictating their behavior without resort to numerical computations, which are of course subject to the choice of discretization scheme and parameters. The continuous-time model is of crucial importance to the evaluation of the approximation efficacy of discrete models and the analysis of their convergence when the time grid is continuously refined. The study of continuous-time models, however, have been somewhat ignored in the current literature, with only a few references available \citep{Ban, Ma2014a, Ma2014b}; and this paper contributes to this line of research by enhancing the understanding of the roles played by continuous-time models and their relationship with well-known discrete models.

\subsection{Qualitative analyses of the network model}
Over the past two decades the traffic modeling community has seen a substantial growth of literature on modeling realistic traffic network phenomenon while maintaining a reasonable computability of the models; the reader is referred to the beginning of the introduction for a review of these results. Two observations are made based on these results: 1) most network models are discrete-time in nature; and 2) very few qualitative properties of these models, such as solution existence, uniqueness and well-posedness, have been analyzed. On the other hand, continuous-time models have recently received increased attention in  dynamic traffic network modeling, and {\it dynamic traffic assignment} (DTA), due to the still emerging mathematical paradigm of {\it differential variational inequality} (DVI)  and the associated computational tools not available in finite-dimensional analyses. The reader is referred to \cite{FKKR, FHNMY, FM2013} and \cite{HZF}  for a series of developments in network modeling that are facilitated by the DVI theory.

\subsubsection{Existence of solutions}

The existence of a continuous-time network solution is far more difficult to establish than when the model is presented in discrete time through straightforward bookkeeping (examples of these discrete-time models include the cell transmission model and the link transmission model). While it is well established that a single conservation law with compatible initial and boundary conditions has a unique entropy solution (we refer the reader to \cite{Bbook} for more discussions on the entropy solutions to scalar conservation laws), a network of road segments, which is viewed as a system of conservation laws with coupling boundary conditions, does not always admit a solution in the sense to be articulated later in Section \ref{secneLWR}.  A common technique for showing the existence and uniqueness of a solution of the network model is called the {\it wave front tracking} (WFT) method. This method is originally proposed by \cite{Dafermos} to constructs piecewise constant approximations of the weak solution to the scalar conservation law of the form \eqref{LWRPDE}. This is done by considering a piecewise affine fundamental diagram and piecewise constant initial/boundary conditions.  One of the primary purposes of the WFT method is to show existence of weak solutions by  means of successive refinements of the aforementioned approximations.  This procedure will be briefly explained in Section \ref{subsecwft}; we refer the reader to \cite{HR2002} for further elaborations. The wave front tracking algorithm is later used by \cite{CGP} to show existence of weak solutions on a road network; however, that existence result is established for a limited junction model in which the number of incoming links must not exceed the number of outgoing links. As we explain in Section \ref{secexistence}, the existence of solutions on the network depends heavily on the type of junction model and Riemann Solver employed. Therefore, the existence result in \cite{CGP} does not cover the merge junction considered in this paper, and we will fill this gap in this paper by providing existence result for networks consisting of the merge and diverge junctions.

In general, rigorous proofs of existence results, usually involving the wave front tracking method \citep {CGP}, are quite cumbersome. Notably, \cite{GP2009} provide sufficient conditions for the existence of network solutions by specifying a set of abstract properties to be satisfied by the junction models (Riemann Solvers). These sufficient conditions will be illustrated in detail and checked against the two specific junction models considered in this paper.

\subsubsection{Well-posedness and uniqueness of solutions}

Another important property of the network model investigated in this paper is well-posedness. A solution is well-posed if it changes continuously with respect to the given initial/boundary conditions of a network. Notice that the uniqueness of a solution, provided that it exists, should come as a natural consequence of well-posedness. Well-posedness is a desirable property for analytical models since it is difficult to understand and predict the behavior of an ill-posed system, it is even more challenging to numerically compute a solution that reflects the dynamics of the system in a meaningful way if well-posedness is not granted.

It should not come as a surprise that the well-posedness property, like existence, should depend on the specific junction type and model employed. This paper establishes well-posedness and solution uniqueness for traffic networks consisting of the two specific junctions to be elaborated in Section \ref{sectwors}. Such a result is established through the notion of generalized tangent vector, which is detailed in Section \ref{secgtv}. In particular, we provide a general framework for analyzing the continuous dependence of the weak solution at a junction on the initial conditions. And such a framework is applicable to other existing junction models and/or Riemann Solvers, leading to a class of well-posedness results for more general network models.

It should be noted that the well-posedness property is closely related to the continuity of the delay operator, which arises from the investigation of {\it dynamic traffic assignment} (DTA) or {\it dynamic user equilibrium} (DUE) problems. The delay operator, with some variations in terminology, refers to a nonlinear mapping between two subsets of certain Hilbert space; this mapping associates to a set of path departure rates its corresponding set of path travel times. Continuity of the delay operator is crucial for the existence of DUE \citep{existence} and convergence of DUE algorithms \citep{FKKR}. Evaluation of the delay operator is done through the {\it dynamic network loading} (DNL) procedure, which is performed on a network with given origin-destination pairs and established route choices of drivers, and is subject to flow propagation constraints, flow conservation constraints, and the {\it first-in-first-out} (FIFO) principle. Since a small change in the path departure rate(s) is propagated throughout the network via the link model (LWR model) and the junction model, it is crucial to study the well-posedness of both models. The former has been well established in the {\it partial differential equations} literature \citep{Bbook}; this paper aims at addressing the latter and provides some preliminary results and insights. However, it should be noted that  this paper deals only with junction models that assume constant vehicle turning rates; while the junction model in an DNL procedure relies on time-varying and endogenous vehicle turning ratios since each vehicle has a prescribed route to follow. Thus, well-posedness of the DNL model, and hence the continuity of the delay operator, cannot be directly analyzed using the proposed framework and will be pursued in another paper.

\subsection{Organization}
The rest of this paper is organized as follows. Section \ref{secneLWR} articulates the network extension of the classical LWR model, and describes in detail two types of junction models to be considered later in this paper. Section \ref{LLWR} presents a derivation of the continuous-time link-based kinematic wave model (LKWM) based on the variational theory. We also show that its time discretization leads to Yperman's link transmission model. Section \ref{secexistence} establishes global existence result for the network-based LWR model in connection with the two specific junction models. In Section \ref{secwellposedness}, we investigate the well-posedness of the junction models. Finally, Section \ref{secNumerical} provides a numerical result of the LKWM and Section \ref{secconclusion} offers some concluding remarks.

\section{Network extension of the LWR model}\label{secneLWR}
In this section, we recap the network extension of the scalar conservation law model originally articulated in \cite{HR1995}. This section provides precise definitions of a weak solution on the network and the two Riemann Solvers employed in this paper, which will be the central focus of our qualitative analyses presented later in Section \ref{secexistence} and Section \ref{secwellposedness}.

We consider a network consisting of only one junction with $m$ incoming links $I_1,\ldots, I_m$ and $n$ outgoing links $I_{m+1},\ldots, I_{m+n}$. For $k=1,\,\ldots,\,m+n$, let $\rho_k(t,\,x)$ be the density on link $I_k$.  On each link the traffic dynamic is governed by the LWR equation \eqref{LWRPDE}. In analogy to a weak solution of a single scalar conservation law \footnote{In the case of a single conservation law \eqref{LWRPDE}, a solution cannot be defined in the classical sense because $\rho(t,\,x)$ may be discontinuous due to the presence of shock waves. Instead, an alternative solution, called the weak solution, is defined through integrals. See \cite{Bbook} and \cite{Evans} for more details}, we define a weak solution at this junction. To do this, we define a set of test functions $\{\phi_k(t,\,x)\}_{k=1}^{m+n}$, such that each $\phi_k(t,\,x)$ is compactly supported in  $(0,\,+\infty)\times (a_k,\,b_k)$. In addition, such a set of functions are smooth across the junction; that is, 
$$
\phi_i(t,\, b_i)~=~\phi_j(t,\,a_j),\quad {\partial \phi_i\over \partial x}(t,\,b_i)~=~{\partial\phi_j\over\partial x}(t,\,a_j),\qquad \forall~~t\in(0,\,+\infty)
$$
\noindent for any $i=1,\ldots, m,~~j=m+1,\ldots, m+n$. 
\begin{definition}{\bf (Weak solutions)}
A collection of densities $\{\rho_{k}(t,\,x)\}_{k=1}^{m+n}$ are called  weak solutions at this junction if, for any set of test functions $\{\phi_k(t,\,x)\}_{k=1}^{m+n}$, the following holds
\begin{equation}\label{weakdef}
\sum_{k=1}^{m+n}\int_{0}^{+\infty}\int_{a_k}^{b_k}\left[\rho_k(t,\,x){\partial \over \partial t}\phi_k(t,\,x) + f_k\big(\rho_k(t,\,x)\big){\partial \over \partial x}\phi_k(t,\,x)\right]dx\,dt~=~0
\end{equation}
\end{definition}

\noindent Any weak solution $\{\rho_k(t,\,\cdot)\}_{k=1}^{m+n}$ satisfies the Rankine-Hugoniot condition:
\begin{equation}\label{juncons}
\sum_{i=1}^m f_i\big(\rho_i(t,\,b_i)\big)~=~\sum_{j=m+1}^{m+n} f_j\big(\rho_j(t,\,a_j)\big)
\end{equation}
which is a generalization of the R-H condition stated for the scalar conservation law \citep{Bbook}. Eqn.  \eqref{juncons} essentially represents the conservation of flow across the junction. Notice that \eqref{juncons} alone is not sufficient to isolate a unique weak solution for the junction, and thus additional conditions are needed. Such additional conditions are usually articulated through the Riemann Solvers, to be presented below.

\subsection{Definition of the Riemann Solvers}

The weak solution at a junction is best analyzed through a simplified problem, known as the {\it Riemann problem}, which is the building block for a complete junction model illustrated in the previous section. A Riemann problem at a junction is an initial value problem with constant density value on each of the incident links. Solution of the Riemann problem is determined by a Riemann Solver, which, for a given Riemann (constant) initial data at the junction, produces a set of boundary conditions so that a unique weak solution at the junction can be obtained by solving an initial-boundary value problem on each of the links.

We present in the following a more precise definition of the Riemann Solvers. Consider the following Riemann initial data: 
$$
\rho_k(0,\, x)~\equiv~\widehat \rho_k\in \left[0,\,\rho_{k}^{jam}\right]\qquad \forall x\in (a_k,\,b_k), \qquad  k\in\{1,\,\ldots,\,m+n\}
$$
where $(\widehat \rho_1,\,\ldots,\,\widehat \rho_{m+n})$ is a tuple of constant densities, and $\rho_k^{jam}$ denotes the jam density on link $I_k$. A Riemann Solver determines another constant tuple $(\overline \rho_1,\,\ldots,\,\overline\rho_{m+n})$ based on the Riemann initial data $(\widehat \rho_1,\,\ldots,\,\widehat \rho_{m+n})$ such that the weak solution at this junction is obtained by solving an initial-boundary value problem 
\begin{equation}\label{ibprob}
\begin{cases}
\partial_t \rho_i(t,\,x)+\partial_x f\big(\rho_i(t,\,x)\big)~=~0
\\
\rho_i(0,\,x)~=~\widehat\rho_i \qquad  \forall x\in(a_i,\,b_i)
\\
\rho_i(t,\,b_i)~=~\overline \rho_i \qquad  \forall t >0
\end{cases}
\qquad\hbox{and}\qquad
\begin{cases}
\partial_t \rho_j(t,\,x)+\partial_x f\big(\rho_j(t,\,x)\big)~=~0
\\
\rho_j(0,\,x)~=~\widehat\rho_j \qquad  \forall x\in(a_j,\,b_j)
\\
\rho_j(t,\,a_j)~=~\overline \rho_j \qquad  \forall t >0
\end{cases}
\end{equation}
on each link, where $i=1,\,\ldots,\,m$, $j=m+1,\,\ldots,\,m+n$. As a result, we have the following definition for the Riemann Solver.

\begin{definition}\label{chapDNL:lwrrsdef} {\bf (Riemann Solver)} A Riemann Solver for the junction  with $m$ incoming links and $n$ outgoing links is a mapping (here $\Pi$ represents the product of indexed quantities)
\begin{align*}
RS:~~\prod_{i=1}^{m+n} \left[0,\,\rho^{jam}_{i}\right] &~\longrightarrow~ \prod_{i=1}^{m+n}\left[0,\,\rho^{jam}_{i}\right]
\\
\big(\widehat\rho_{1},\,\ldots,\, \widehat\rho_{m+n}\big)&~\mapsto~\big(\overline\rho_1,\,\ldots,\,\overline\rho_{m+n}\big)
\end{align*}
so that the following holds:
\begin{enumerate}
\item[(i)] The weak solution at this junction consists of solutions of the initial-boundary value problems \eqref{ibprob}. 

\item[(ii)] The Rankine-Hugoniot condition holds:
$$
\sum_{i=1}^m f_i\big(\overline\rho_i\big)~=~\sum_{j=m+1}^{m+n}f_j\big(\overline\rho_j\big)
$$
\item[(iii)] The {\it consistency condition} holds:
$$
RS\big[RS[\widehat\rho]\big]~=~RS[\widehat\rho]
$$
\end{enumerate}

\end{definition}

\noindent A Riemann Solver usually involves certain driving behavior or control and management considerations. Examples include vehicle turning percentage \citep{CTM2}, driving priority or right-of-way \citep{CGP, CTM2}, and signal control \citep{HGPFY}. We will next introduce two specific Riemann Solvers, which will be the focus of our qualitative analyses presented later.

\subsection{Two specific Riemann Solvers}\label{sectwors}

We consider two simple junctions depicted in Figure \ref{figtwojunc}. We note that these two junction types are fundamental to the construction of more complicated junction types, and the modeling techniques employed for these two junctions are generalizable to junctions with different topologies.  
\begin{figure}[h!]
\centering
\includegraphics[width=0.7\textwidth]{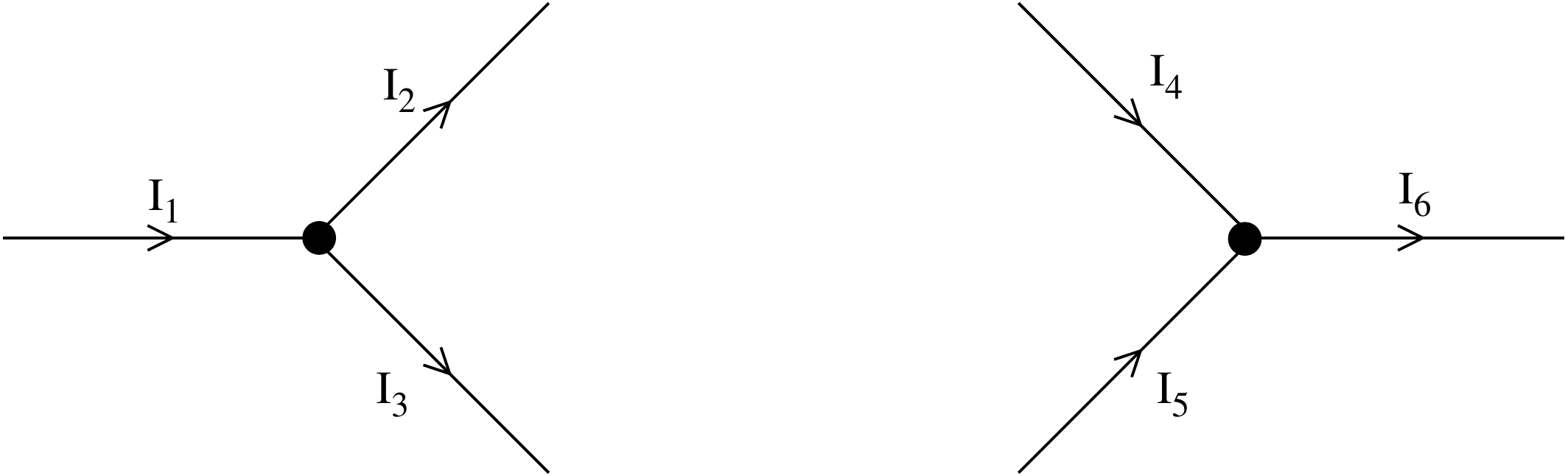}
\caption{The diverging (left) and merge (right) junctions.}
\label{figtwojunc}
\end{figure}

The definition of a Riemann Solver is greatly facilitated by the notions of demand and supply \citep{LK1999}. For each incoming link, we define its demand, denoted by $D(t)$, to be the maximum flow at which cars can leave this link. It is determined by the density value close to the exit of this link:
\begin{equation}\label{demanddef}
D(t)~=~\begin{cases}
C  \quad & \hbox{if}~~ \rho(t,\,b-)~\geq~\sigma 
\\
f\big(\rho(t,\,b-)\big)\quad & \hbox{if}~~ \rho(t,\,b-)~<~\sigma
\end{cases}
\end{equation}
where $C$ denotes the flow capacity of the link, $\sigma$ is the unique density at which the flow is maximized. Similarly, for each outgoing link, we define its supply $S(t)$ to be the maximum flow at which cars can enter this link:

\begin{equation}\label{supplydef}
S(t)~=~\begin{cases}
C  \quad & \hbox{if}~~ \rho(t,\,a+)~<~\sigma
\\
f\big(\rho(t,\,a+)\big)  \quad & \hbox{if}~~ \rho(t,\,a+)~\geq~\sigma
\end{cases}
\end{equation}

\subsubsection{Riemann Solver for the diverge junction}\label{secdiverge}

Consider the diverge node shown in Figure \ref{figtwojunc}, with one incoming link $I_1$ and two outgoing links $I_2$ and $I_3$. As usual, each $I_i$ is expressed as a spatial interval $[a_i,\,b_i]$; the  fundament diagram is $f_i(\cdot)$ with the flow capacity $C_i$. The demand function $D_1(t)$  and the supply functions, $S_2(t)$ and $S_3(t)$, are defined by \eqref{demanddef} and \eqref{supplydef} respectively.

The diverge junction model requires the knowledge of a (constant) vehicle turning ratios $\alpha_{1,2},\,\alpha_{1,3}>0$ such that $\alpha_{1,2}+\alpha_{1,3}=1$.  Notice that while this paper treats these turning ratios as constants (time-independent), they can be time-varying and may be determined either exogenously via empirical data, or endogenously via drivers' prescribed route choices. The latter is particularly relevant to the notion of {\it dynamic network loading} arising from {\it dynamic traffic assignment} study. The Riemann Solver of interest is first proposed in \cite{CTM2} based on two assumptions:
\begin{itemize}
\item[(A1)] Vehicles leaving the incoming link distribute to the outgoing links according to some given turning ratio.
\item[(A2)] Subject to  (A1), the flow through the junction is maximized.  
\end{itemize}
Moreover, we have the following obvious constraints:
\begin{equation}\label{qD}
q_{out,1}\leq D_1(t),\qquad q_{in,2}(t)\leq S_2(t), \qquad q_{in,3}(t)\leq S_3(t)
\end{equation} 
Solving (A1), (A2) and \eqref{qD} together yields
\begin{equation}\label{fifodiv}
\begin{array}{l}
q_{out,1}(t)~=~\displaystyle\min\left\{D_1(t),~{S_2(t)\over \alpha_{1,2}},~{S_3(t)\over\alpha_{1,3}}\right\}
\\
q_{in,2}(t)~=~\alpha_{1,2} \cdot q_{out,1}(t)
\\
q_{in,3}(t)~=~\alpha_{1,3}\cdot q_{out,1}(t)
\end{array}
\end{equation}
where $q_{out,1}(t)$ denotes the exit flow on link $I_1$, and $q_{in,i}(t)$ denotes the inflow on link $I_i$, $i=2,\,3$.

\subsubsection{Riemann Solver for the merge junction}\label{secmerge}

Consider the merge junction in Figure \ref{figtwojunc}, with two incoming links $I_4$ and $I_5$ and one outgoing link $I_6$. The demand and supply functions for these links are defined as before. In view of this simple merge junction, assumption (A1) becomes irrelevant; and assumption (A2) cannot determine a unique solution. More generally, as pointed out by \cite{CGP}, when the number of incoming links exceeds the number of outgoing links, (A1) and (A2) combined are not sufficient to find a unique solution. To address this issue, we invoke a right-of-way parameter $p\in(0,\,1)$ and the following driving priority rule:\\

\noindent (R1) The actual link exit flows, $q_{out,4}$ and $q_{out,5}$, satisfy 
$$
p\cdot q_{out,4}~=~q_{out,5}
$$
\noindent for some right-of-way parameter $p\in(0,\,1)$.\\

\noindent However, notice that the rule (R1) may be incompatible with assumption (A2). See Figure \ref{figmerge} for an illustration.

\begin{figure}[h!]
\centering
\includegraphics[width=0.75\textwidth]{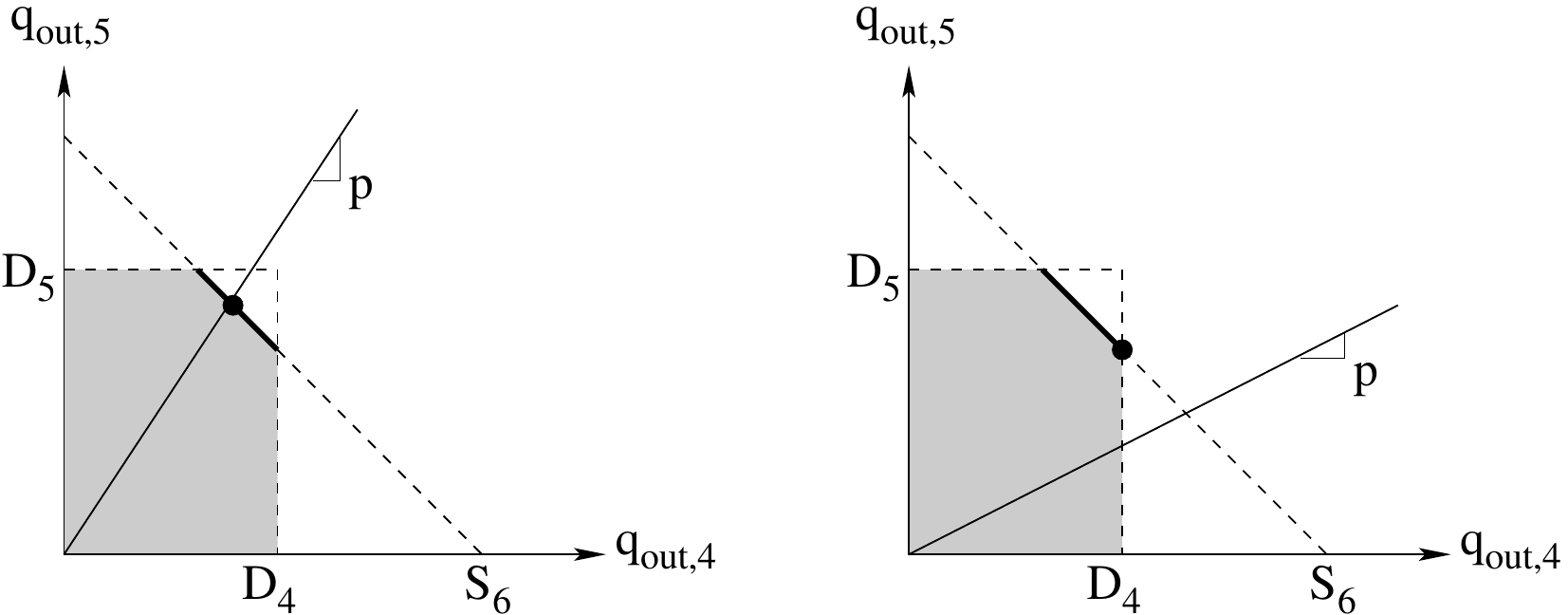}
\caption{Left: Rule (R1) is compatible with (A2). There exists a unique point satisfying both (A2) and (R1). Right: Rule (R1) is incompatible with (A2); in this case, the Riemann Solver selects, among all points that maximizes the flow, the point that is closest to the line through the origin with slope $p$.}
\label{figmerge}
\end{figure}

Whenever there is a conflict between (R1) and (A2), we will respect (A2) and relax (R1) so that the solution is chosen to be the point that approximates (R1) best among the points yielding the maximum flow. Moreover, such a solution is unique. More precisely, we let $\Omega$ be the set of 2-tuples $(q_{out,4},\,q_{out,5})$ that solves the following maximization problem:
$$
\max q_{out,4} + q_{out,5}
$$
such that 
$$
q_{out,4}\in[0,\,D_4],\qquad q_{out,5}\in[0,\,D_5],\qquad q_{out,4}+q_{out,5}~\leq~S_6
$$
Moreover, we define the set $P\doteq \{(q_{out,4},\,q_{out,5})\in\mathbb{R}_+^{2}:~p\cdot q_{out,4}=q_{out,5}\}$. Then, the solution of the merge junction model is defined to be the projection of $P$ onto $\Omega$; that is,
\begin{equation}\label{RSmergedef}
(q^*_{out,4},\,q^*_{out,5})~=~\underset{(q_{out,4},\,q_{out,5})\in\Omega} {\hbox{argmin}} D\left((q_{out,4},\,q_{out,5}),\,P\right)
\end{equation}
where $D\left((q_{out,4},\,q_{out,5}),\,P\right)$ denotes the distance of a point from a set $P$ and is defined to be 
$$
D\left((q_{out,4},\,q_{out,5}),\,P\right)~=~\min_{(x,\,y)\in P}\left\|(x,\,y)-(q_{out,4},\,q_{out,5})\right\|
$$ 
and the above norm is in the Euclidean sense.

\begin{remark}\label{remarkmergemono}
An inspection of Figure \ref{figmerge} reveals that the flow through the merge junction is monotonically increase with respect to $D_4$, $D_5$ and $S_6$.
\end{remark}

\section{Formulation of the  link-based kinematic wave model on networks}\label{LLWR} 
We derive in this section the continuous-time kinematic wave model (LKWM) on a road network, by assuming the triangular fundamental diagram. A more general LKWM with arbitrary concave fundamental diagram is presented in another paper \citep{spillback}. As we explain, the proposed LKWM does not involve any density variables; instead, it employs vehicle flows with some binary variables to capture the propagation of kinematic waves through junctions. The derivation of the LKWM employs the Lax-Hopf formula \citep{ABP, VT1, VT2}, which captures the interaction of a separating shock on a link with the boundaries of that link and its adjacent junctions. We make note of the fact that a forward time-discretization of the LKWM leads to the {\it link transmission model} (LTM) \citep{LTM}. This will be shown in Section \ref{subsecLTM}.

\subsection{Definition of the primary variables}\label{statevariables} 

We consider a homogenous link represented by an interval $[a,\,b]$, with $b-a=L$ being the length of this link. In the formulation of the LKWM, we consider a binary variable $r(t,\,x)$, $(t,\,x)\in[0,\,+\infty)\times[a,\,b]$ which takes value zero if the traffic at $(t,\,x)$ is in the free-flow phase, which corresponds to the left half of the fundamental diagram (see Figure \ref{figfd1}); and takes value one if the traffic is in the congested phase, which corresponds to the right half of the fundamental diagram.

It is obvious from Figure \ref{figfd1} that a single flow value $q$ corresponds to two different density values and two different traffic states: (1) the free-flow phase ($r=0$), and (2) the congested phase ($r=1$).  Therefore, a pair $\big(q(t,\,x),\,r(t,\,x)\big)\in [0,\,C]\times \{0,\,1\}$ determines a unique traffic state. As a consequence, the following map is well-defined.
\begin{equation}\label{psidef}
\psi(\cdot): [0,\,C]\times \{0,\,1\}~\rightarrow ~ [0,\,\rho^{jam}],\qquad  ( q,\,r)\mapsto \rho
\end{equation}

\begin{figure}[h]
\centering
\includegraphics[width=0.45\textwidth]{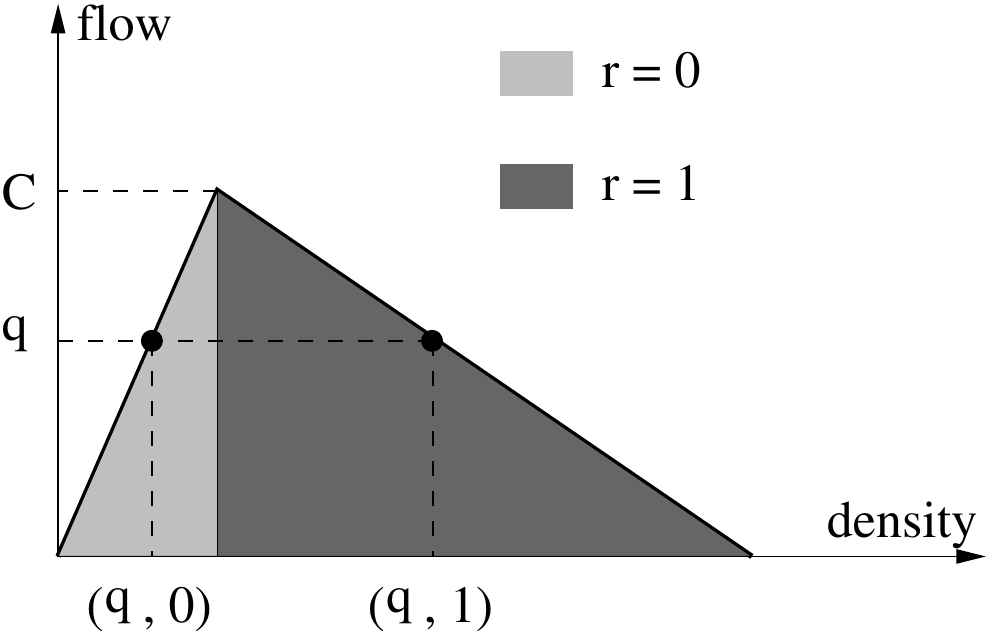}
\caption{Illustration of the variables $( q,\,r)$. }
\label{figfd1}
\end{figure}

\subsection{\label{shock} Shock formation and propagation within the link}
In section \ref{sectwors}, we express the Riemann Solvers for a merge and a diverge junction in terms of link demand and supply, which are effectively dependent on the binary variables $r$ and the flow variables $q$. The goal of this subsection is to establish the dynamics for $r$ and $q$, through the introduction of a separating shock and the Lax-Hopf formula.

The separating shock is a generalized characteristic \citep{Dafermos}), which divides a link into two parts: the free-flow region where $r=0$ and the congested region where $r=1$.  We note the fact that if a link is initially empty, i.e. $\rho(0,\,x)\equiv 0$, then there can be at most one separating shock present inside this link.

\begin{lemma}\label{sshocklemma}
Given a link expressed as a spatial interval $[a,\,b]$, consider a weak entropy solution on this link expressed using the new variables $\big(q(t,\,x),\,r(t,\,x)\big)$ with $q(0,\,x)\equiv 0$ and $r(0,\,x)\equiv 0$ for all $x\in[a,\,b]$. Then, the following statement holds.
\begin{enumerate}
\item For every $t\geq 0$, there exists at most one $x^*(t)\in(a,\,b)$ such that $r(t,\,x^*(t)-)<r(t,\,x^*(t)+)$
\item For all $x\in(a,\,b)$, 
\begin{align*}
&r(t,\, x)~=~0,\qquad \hbox{if}~~ x~<~x^*(t)\\
&r(t,\,x)~=~1,\qquad \hbox{if}~~ x~>~x^*(t)
\end{align*}
Thus, $x^*(t)$ represents the location of the separating shock.
\end{enumerate}
\end{lemma}
\begin{proof}
See \cite{Bretti et al}.
\end{proof}

\begin{remark}
The uniqueness of a separating shock can be assured with conditions slightly less restrictive than zero initial condition; that is, when there exists at most one connected free-flow region and at most one connected congested region at $t=0$. These results also apply to other types of fundamental diagrams. 
\end{remark}

According to Lemma \ref{sshocklemma}, if the link is initially empty, the separating shock emerges from the downstream boundary $x=b$ at certain point in time, and propagates towards the interior of the link. The speed of this separating shock is given by the  Rankine-Hugoniot condition \citep{Evans}, which states that the speed of the shock is equal to the ratio of the jump in flow and the jump in density across the shock; see Figure \ref{figsshock} for an illustration of this condition.
\begin{figure}[h!]
\centering
\includegraphics[width=.7\textwidth]{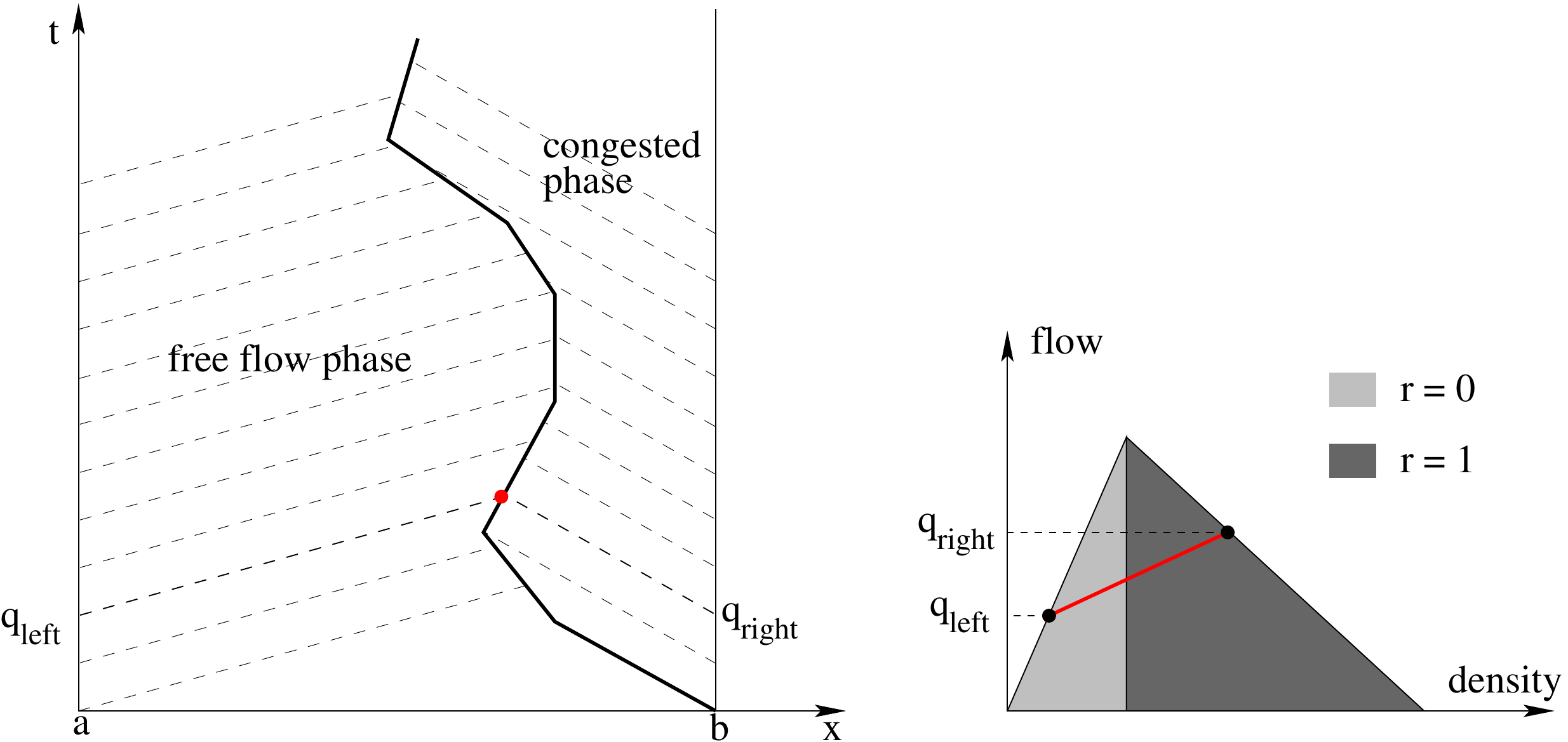}
\caption{Example of the separating shock. Left figure: the temporal-spatial domain separated by the separating shock. Right figure: the speed of the separating shock, ${d\over dt}x^*(t)$, is given by the Rankine-Hugoniot condition, i.e. the slope of the line segment on the fundamental digram that connects $(q_{left},\,0)$ and $(q_{right},\,1)$. }
\label{figsshock}
\end{figure}

It is clear that as long as the separating shock remains in the interior of the link $I_i$, the entrance of the link remains in the free-flow phase and the exit remains in the congested phase. As a result, the link demand and supply remain constants and are equal to the flow capacity. In this case, there is no interaction between the two boundaries of the same link, nor is there interaction between the upstream junction and the downstream junction of this link. 

On the other hand, if the separating shock reaches either boundary, for which we call the {\it latent shock}, the two boundaries become interdependent. More specifically, we distinguish between the following two extreme cases.
\begin{itemize}
\item[i)] The separating shock reaches the downstream boundary, i.e. $x^*(t)=b$. In this case, the link is entirely in the free-flow phase. Consequently, a characteristic line with slope $k$ connects the left and the right boundaries (see Figure \ref{figextreme}), where $k$ denotes the speed of forward-propagating waves. Since the density value along a characteristic line remains constant, we have that
\begin{equation}\label{caseii}
q\left(t,\,b-\right)~=~q\left(t-{L \over v},\, a\right)
\end{equation}
where $L$ denotes the length of the link. 
 
\item[ii)] the shock reaches the upstream boundary, i.e. $x^*(t)=a$. In this case,  the link is dominated by the congested phase, and the condition at the downstream boundary directly affects the flow at the upstream boundary (see Figure \ref{figextreme}).
\begin{equation}\label{casei}
q\left(t,\,a+\right)~=~q\left(t-{L\over w},\, b\right)
\end{equation}
where $w$ denotes the speed of backward-propagating waves.    
\end{itemize}

\begin{figure}[h!]
\centering
\includegraphics[width=.75\textwidth]{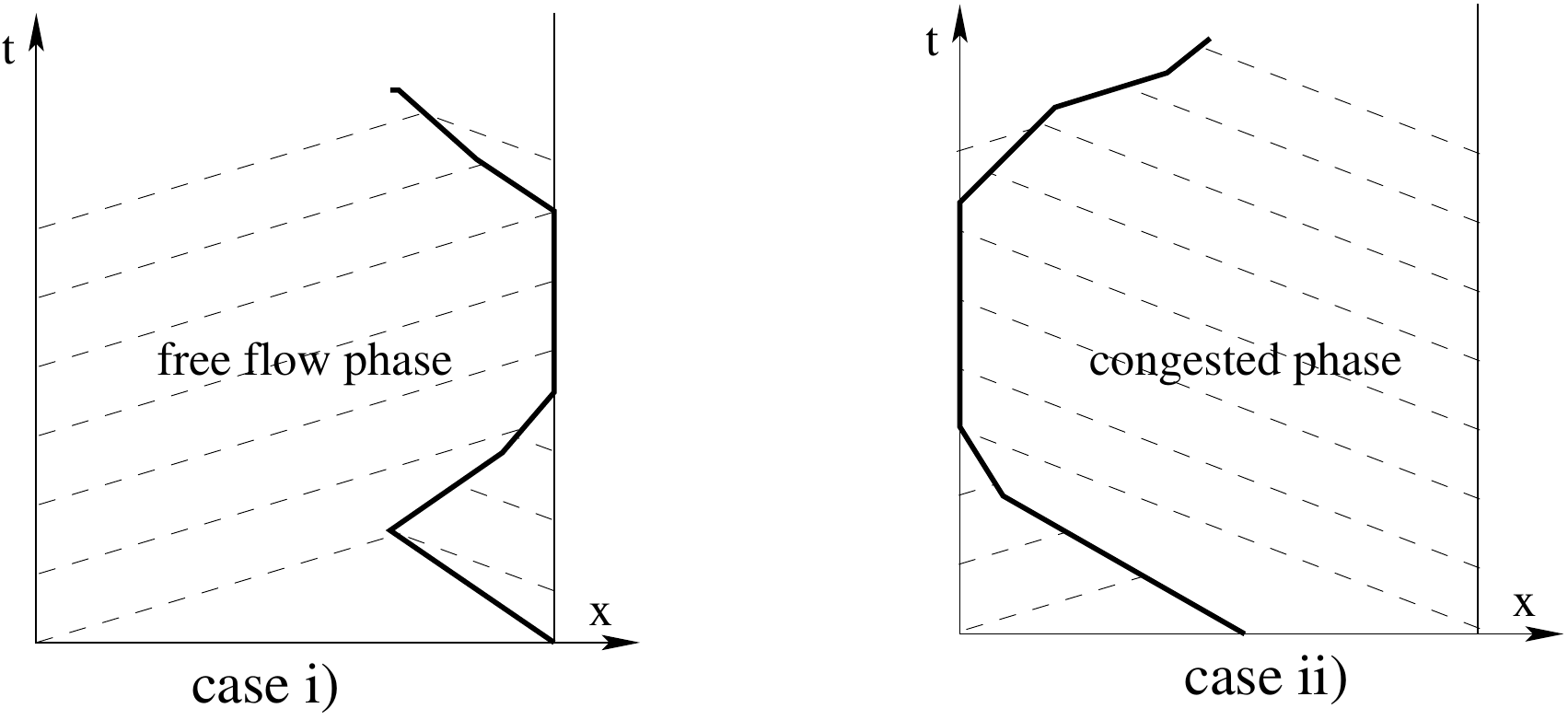}
\caption{Examples of latent separating shocks. Left: the separating shock reaches the right (downstream) boundary. Right: the separating shock reaches the left (upstream) boundary.}
\label{figextreme}
\end{figure}

In the presence of latent shocks, either the demand or supply of the link is determined by the flow variable at the other end of this link. When this happens, the interaction between the two boundaries, and hence the two junctions, begin to take place. In addition, the demand and supply of the link are no longer constants; as a result, the upstream junction and the downstream junction cannot be handled by their respective Riemann Solvers without influences from each other. Simply put, different junctions in the network may 'communicate' to each other in the current model, which is the main reason for the propagation of congestion through the network, queue spillback, and gridlock.

Due to the important role played by the aforementioned two extreme cases, we will next provide a method for detecting the presence of the latent separating shock using a variational method known as the Lax-Hopf formula.

\subsection{The Lax-Hopf formula approach for detecting latent shocks}\label{secLax}

The Lax-Hopf formula provides a variational formulation of solutions of the Hamilton-Jacobi equations, which describes the evolution of cumulative vehicle curves. The Lax-Hopf formula has several variations in the literature and can be independently derived from the theory of calculus of variations \citep{Evans}, the viability theory \citep{ABP, CC1}, and traffic flow theory \citep{VT1, Newella}.

Let us begin by introducing the Moskowitz function \citep{Moskowitz} or the N-curve \citep{Newella}, denoted by $N(\cdot,\,\cdot): [0,\,+\infty)\times[a,\,b]\rightarrow \Re_+$, such that
\begin{equation}\label{Ndef}
{\partial\over \partial x} N(t,\,x)~=~-\rho(t,\,x),\qquad{\partial\over \partial t} N(t,\,x)~=~f\big(\rho(t,\,x)\big)
\end{equation}
The function $N(t,\,x)$ measures the cumulative number of vehicles that have passed location $x$ by time $t$. The function $N(\cdot,\,\cdot)$ satisfies the following Hamilton-Jacobi equation:
\begin{equation}\label{HJeqn}
{\partial\over \partial t} N(t,\,x)-f\left( -{\partial\over\partial x} N(t,\,x)\right)~=~0\qquad (t,\,x)\in[0,\,+\infty)\times [a,\,b]
\end{equation}
where $f(\cdot)$ denotes the fundamental diagram, which, in the remainder of this section, is assumed to be triangular:
\begin{equation}\label{deltafd}
f(\rho)~=~\begin{cases}
v\,\rho\qquad & \rho\in[0,\,\sigma]\\
-w\,(\rho-\rho^{jam})\qquad & \rho\in(\sigma,\,\rho^{jam}]
\end{cases}
\end{equation}
where $v>0$ is the speed of forward-propagating waves, and $w>0$ is the speed of backward-propagating waves; $\sigma$ denotes the unique density at which the flow is maximized. We further introduce the concave transformation of this  fundamental diagram.
\begin{equation}\label{legendre}
f^*(u)~\doteq~\sup_{\rho\in[0,\,\rho^{jam}]}\left\{f(\rho)-u\,\rho\right\}~=~C-\sigma u \qquad u\in[-w,\,v]
\end{equation}

Before we state the Lax-Hopf formula, we define the upstream and downstream boundary conditions that are consistent with the assumption of zero initial condition.

\begin{definition}{\bf (Boundary conditions for the H-J equation (\ref{HJeqn}))}\label{bdyconditions}
Consider a link with flow capacity $C$. The upstream boundary condition $N_{up}(\cdot): [0,\,+\infty)\rightarrow \Re_+$ and the downstream boundary condition $N_{down}(\cdot): [0,\,+\infty)\rightarrow \Re_+$ are both non-decreasing, Lipschitz continuous functions with Lipschitz constants $C$, and they satisfy $N_{up}(0)=N_{down}(0)=0$. 
\end{definition}

The last condition in this definition is to ensure zero initial condition on the link; that is, the link is initially empty. Notice that in Definition \ref{bdyconditions}, both conditions are imposed in the strong sense; that is, $N_{up}(\cdot)$ and $N_{down}(\cdot)$ are the time-integrals of the link inflow and outflow respectively. These are in contrast to the weak boundary conditions discussed by \cite{ABP} and \cite{HGPFY}. The next result is a special case of the Lax-Hopf formula with a triangular fundamental diagram and zero initial condition. Its proof is derived and presented in a number of ways by \cite{Newella, CC2, VT1} and \cite{HGPFY}.

\begin{proposition}\label{LHprop} {\bf (The Lax-Hopf formula)}
Given upstream and downstream boundary conditions $N_{up}(\cdot)$ and $N_{down}(\cdot)$ as in Definition \ref{bdyconditions}, the solution of the H-J equation \eqref{HJeqn} is given by
\begin{equation}\label{simpleLH}
N(t,\,x)~=~\min\left\{N_{up}\left(t-{x-a\over v}\right)\,,\quad N_{down}\left(t-{b-x\over w}\right)+\rho^{jam}(b-x)   \right\}
\end{equation}
$\forall  (t,\,x)\in[0,\,+\infty)\times[a,\,b]$.
\end{proposition}

\begin{remark}\label{ssremark}
In the expression of the solution $N(t,\,x)$, the two expressions to be compared on the right hand side are closely related to the location of the separating shock. Namely, for an arbitrary point $(t,\,x)$ in the domain,  if $N_{up}\left(t-{x-a\over v}\right)$ is strictly less than $N_{down}\left(t-{b-x\over w}\right)+\rho^{jam}(b-x)$, then the upstream boundary condition is active at $(t,\,x)$, which means that a characteristic line connects $(t,\,x)$ with some point on the left boundary (see case i) of Figure \ref{figextreme}). Similarly, if $N_{up}\left(t-{x-a\over v}\right) > N_{down}\left(t-{b-x\over w}\right)+\rho^{jam}(b-x)$, then a characteristic line connects $(t,\,x)$ with some point on the right boundary (see case ii) of Figure \ref{figextreme}). 
\end{remark}

Based on the observation provided in Remark \ref{ssremark}, we have the following result.

\begin{proposition}\label{propdetection}
Let $N(\cdot,\,\cdot): [0,\,+\infty)\times[a,\,b]$ be the unique solution of the Hamilton-Jacobi equation (\ref{HJeqn}) given by the Lax-Hopf formula \eqref{simpleLH} with upstream and downstream boundary conditions $N_{up}(\cdot),\,N_{down}(\cdot)$. Then the following statements hold:

\noindent 1.  For any $(t,\,x)\in[0,\,+\infty)\times[a,\,b]$,
\begin{align}
\label{latent1}
x~\leq~x^*(t)\quad \Longleftrightarrow \quad N_{up}\left(t-{x-a\over v}\right)~\leq~N_{down}\left(t-{b-x\over w}\right)+\rho^{jam}(b-x)\\
\label{latent2}
x~\geq~x^*(t)\quad \Longleftrightarrow \quad N_{up}\left(t-{x-a\over v}\right)~\geq~N_{down}\left(t-{b-x\over w}\right)+\rho^{jam}(b-x)
\end{align}
\noindent 2. In particular, for all $t\geq 0$,
\begin{align}
\label{latent3}
x^*(t)&~=~a\quad \Longleftrightarrow \quad N_{up}\left(t\right)~=~N_{down}\left(t-{L\over w}\right)+\rho^{jam}L\\
\label{latent4}
x^*(t)&~=~b\quad \Longleftrightarrow \quad N_{up}\left(t-{L\over v}\right)~=~N_{down}\left(t\right)
\\
\label{latent5}
a<x^*(t)&<b  \quad \Longleftrightarrow \quad  
\begin{cases}
N_{up}\left(t\right)~<~N_{down}\left(t-{L\over w}\right)+\rho^{jam}L
\\
N_{up}\left(t-{L\over v}\right)~>~N_{down}(t)
\end{cases}
\end{align}
\end{proposition}
\begin{proof}
\eqref{latent1} and \eqref{latent2} are both consequences of Remark \ref{ssremark}. That is, a point $x$ is to the left of the separating shock $x^*(t)$ (in other words, traffic at $x$ is in the free-flow phase) if and only if the upstream boundary condition in the Lax-Hopf formula \eqref{simpleLH} is active; on the other hand, the point $x$ is to the right of the separating shock (traffic at $x$ is in the congested phase) if and only if the downstream boundary condition is active. This establishes \eqref{latent1} and \eqref{latent2}.

As a special case of \eqref{latent1} and \eqref{latent2}, we deduce that $x^*(t)=a$ if and only if the downstream condition is active, i.e.,
\begin{equation}\label{latentproof1}
N_{up}\left(t\right)~\geq~N_{down}\left(t-{L\over w}\right)+\rho^{jam}L
\end{equation}
Similarly, $x^*(t)=b$ if and only if the upstream condition is active, i.e.,
\begin{equation}\label{latentproof2}
N_{up}\left(t-{L\over v}\right)~\leq~N_{down}(t)
\end{equation}
We next show that the following always hold: 
\begin{equation}\label{LHconst}
N_{up}(t)~\leq~N_{down}\left(t-{L\over w}\right)+\rho^{jam}L\qquad\hbox{and}\qquad  N_{up}\left(t-{L\over v}\right)~\geq~ N_{down}(t)
\end{equation}
Indeed, by taking $x=a$ and $x=b$ respectively in the Lax-Hopf formula \eqref{simpleLH}, we get
\begin{align*}
N_{up}(t)&~=~N(t,\,a)~=~\min\left\{ N_{up}\left( t\right),\quad N_{down}\left(t-{L\over w}\right)+\rho^{jam}L\right\}
\\
N_{down}(t)&~=~N(t,\,b)~=~\min\left\{ N_{up}\left(t-{L\over v}\right),\quad N_{down}(t)\right\}
\end{align*}
Thus there must hold that $N_{up}(t)\leq N_{down}\left(t-{L\over w}\right)+\rho^{jam}L$ and $N_{up}\left(t-{L\over v}\right)\geq N_{down}(t)$. In view of these two inequalities, conditions \eqref{latentproof1} and \eqref{latentproof2} are revised to be: 
\begin{align*}
x^*(t)&~=~a \quad \Longleftrightarrow\quad N_{up}(t)~=~N_{down}\left(t-{L\over w}\right)+\rho^{jam}L
\\
x^*(t)&~=~b \quad \Longleftrightarrow\quad N_{up}\left(t-{L\over v}\right)~=~N_{down}(t)
\end{align*}
which shows \eqref{latent3} and \eqref{latent4}. Condition \eqref{latent5} follows immediately from \eqref{latent3}, \eqref{latent4}, and \eqref{LHconst}.
\end{proof}

\begin{remark}
The double-queue model recently proposed by \cite{DQM1} relies on the concept of an upstream queue and a downstream queue that may co-exist on the same link. These two concepts corresponds to the first and the second terms in \eqref{simpleLH} respectively; see \cite{Ma2014b} for a continuous-time representation of the double-queue concept and its relationship with the variational representation \eqref{simpleLH}. Thus, the double-queue model may be interpreted as a special case of the LWR model with a triangular fundamental diagram. 
\end{remark}

\subsection{The DAE system}\label{secDAE}
The dynamic of the propagation of kinematic waves within a link has been described in Proposition \ref{propdetection} using some algebraic expressions. This is done by invoking the Lax-Hopf formula and the concept of the separating shock. The goal of this section is to extend our analyses to a network and thereby establish the DAE system formulation of the network-based LKWM.

Let us consider a general traffic network represented by a directed graph $G(\mathcal{A},\,\mathcal{V})$, where $\mathcal{A}$ and $\mathcal{V}$ denote the set of links and nodes in this network, respectively.   We begin by introducing the following notations.
\begin{framed}
\vspace{-0.2 in}
\begin{align*}
q_{in,i}(t): \qquad\qquad &\hbox{the flow at which vehicles enter link }I_i\in\mathcal{A}; \\
q_{out,i}(t): \qquad\qquad &\hbox{the flow at which vehicles exit link }I_i\in\mathcal{A}; \\
D_i(t): \qquad\qquad &\hbox{the demand of link }I_i\in\mathcal{A}; \\
S_i(t): \qquad\qquad &\hbox{the supply of link }I_i\in\mathcal{A}; \\
N_{up,i}(t): \qquad\qquad  & \hbox{the cumulative number of vehicles that have entered link } I_i\in\mathcal{A}; \\
N_{down,i}(t): \qquad\qquad & \hbox{the cumulative number of vehicles that have exited link } I_i\in\mathcal{A}; \\
\alpha_{i,j}:\qquad\qquad & \hbox{the proportion of of traffic leaving link } I_i \hbox{ that is entering link } I_j.
\end{align*}
\vspace{-0.2 in}
\end{framed}

\noindent For each $v\in\mathcal{V}$, denote by $\mathcal{I}^v$ and $\mathcal{O}^v$ the sets of incoming links and outgoing links, respectively. For the compactness of notation, we will encapsulate a given junction model (Riemann Solver) at node $v$ as the following conceptual mapping: 
$$
\left[ \big(q_{out,i}(t)\big)_{i\in\mathcal{I}^v}~,~ \big(q_{in,j}(t)\big)_{j\in\mathcal{O}^v}\right]~=~RS\left[\big(D_i(t)\big)_{i\in\mathcal{I}^v}~,~  \big(S_{j}(t)\big)_{j\in\mathcal{O}^v}\right]
$$
\noindent Note that such Riemann Solvers are expressed in terms of link demand/supply and inflow/outflow. The following DAE system summarizes our analysis of the network model presented so far.
\begin{align}\label{DAE1}
&{d\over dt}N_{up,i}(t)~=~q_{in,i}(t),\qquad {d\over dt}N_{down, i}(t)~=~q_{out,i}(t),\qquad I_i\in\mathcal{A}\\
\label{DAE2}
&D_i(t)=\begin{cases}
q_{in,i}\left(t-{L_i\over v_i}\right)\quad & \hbox{if}~~N_{up,i}\left(t-{L_i\over v_i}\right)=N_{down, i}(t)\\
C_i\quad & \hbox{if}~~N_{up,i}\left(t-{L_i\over v_i}\right)>N_{down, i}(t)
\end{cases},\qquad I_i\in\mathcal{I}^v\\
\label{DAE3}
&S_j(t)=\begin{cases}
q_{out,j}\left(t-{L_j\over w_j}\right)\quad &\hbox{if}~~ N_{up,j}(t)=N_{down,j}\left(t-{L_j\over w_j}\right)+\rho^{jam}_j\,L_j \\
C_j \quad & \hbox{if}~~ N_{up,j}(t)<N_{down,j}\left(t-{L_j\over w_j}\right)+\rho^{jam}_j\,L_j
\end{cases},\quad I_j\in\mathcal{O}^v\\
\label{DAE4}
& \left[ \big(q_{out,i}(t)\big)_{i\in\mathcal{I}^v}~,~ \big(q_{in,j}(t)\big)_{j\in\mathcal{O}^v}\right]~=~RS\left[\big(D_i(t)\big)_{i\in\mathcal{I}^v}~,~  \big(S_{j}(t)\big)_{j\in\mathcal{O}^v}\right]
\end{align}
\begin{equation}\label{DAE6}
q_{in, j}(t) ~=~\sum_{I_i\in \mathcal{I}^v}\alpha_{i,j}\,q_{out, i}(t),\qquad I_j\in \mathcal{O}^v,\quad v\in\mathcal{V}
\end{equation}
In this DAE system, $L_i$ denotes the length of the link $I_i$; $v_i$ and $w_i$ are, respectively,  the forward and backward wave speeds associated with the triangular fundamental diagram; $\rho^{jam}_i$ denotes the jam density of link $I_i$; and $C_i$ denotes the flow capacity of $I_i$. Identity \eqref{DAE1} expresses the relationship between vehicle flow and vehicle count; this is the only place in the DAE system where differentiations are invoked. Equations \eqref{DAE2} and \eqref{DAE3} apply Proposition \ref{propdetection} to determine the free-flow/congested phase at the boundaries of each link, and thereby establish link demand and supply according to \eqref{demanddef}-\eqref{supplydef}. \eqref{DAE4} determines the boundary conditions at each junction based on the information of link demand and supply. Finally, \eqref{DAE6} expresses the flow conservation and re-distribution at each junction.

\subsection{Time-discretization of the DAE system}\label{secdiscreteDAE}
This section provides a discretized version of the proposed DAE system, for the benefit of efficient computation. A numerical example of this discretized DAE system and some visualization of vehicle spillback will be presented in Section \ref{secNumerical}.  

In this section we adopt the naming convention that uses a superscript $k$ to indicate the discrete value of  a variable at the $k$-th time interval, where the time step is indicated as $\Delta t$. Moreover, for each link $I_i\in\mathcal{A}$, we let $\Delta^f_i \doteq \left[{L_i\over v_i \Delta t}\right]$, $\Delta_i^b\doteq \left[{L_i\over w_i\Delta t}\right]$, where $\left[ x\right]$ rounds a real number $x$ to the nearest integer. In other words, $\Delta_i^f$ (or $\Delta_i^b$) represents the number of time intervals needed for the forward (or backward) kinematic wave to traverse the link. The differentiation in \eqref{DAE1} is re-written as time-integration, which is approximated by a simple rectangular quadrature: 
$$
N_{up,i}^k~=~\Delta t \sum_{l=1}^k q_{in,i}^l,\qquad N_{down,i}^k~=~\Delta t \sum_{l=1}^k q_{out,i}^l
$$

\noindent The DAE system \eqref{DAE1}-\eqref{DAE6} can thus be written in discrete time as:
\begin{align}
D_i^k~=~&
\begin{cases}
q_{in,i}^{k-\Delta_i^f} \quad & \hbox{if}~~ \displaystyle \sum_{l=1}^{k-\Delta_i^f}q_{in,i}^l~=~\sum_{l=1}^k q_{out,i}^l
\\
C_i \quad   & \hbox{if}~~ \displaystyle \sum_{l=1}^{k-\Delta_i^f}q_{in,i}^l~>~ \sum_{l=1}^k q_{out,i}^l
\end{cases}
\qquad\qquad I_i\in\mathcal{I}^v
\\
S_j^k~=~&
\begin{cases}
q_{out,j}^{k-\Delta_j^b} \quad & \hbox{if}~~ \displaystyle \Delta t\sum_{l=1}^{k}q_{in,j}^l=\Delta t\sum_{l=1}^{k-\Delta_j^b}q_{out,j}^l +\rho_j^{jam}L_j
\\
C_j \quad & \hbox{if}~~ \displaystyle \Delta t\sum_{l=1}^{k}q_{in,j}^l<\Delta t\sum_{l=1}^{k-\Delta_j^b}q_{out,j}^l +\rho_j^{jam}L_j
\end{cases}
\qquad I_j\in\mathcal{O}^v
\\
&\left[ \big(q^{k+1}_{out,i}\big)_{i\in\mathcal{I}^v}~,~ \big(q^k_{in,j}\big)_{j\in\mathcal{O}^v}\right]~=~RS\left[\big(D^k_i\big)_{i\in\mathcal{I}^v}~,~  \big(S^k_{j}\big)_{j\in\mathcal{O}^v}\right]\qquad\forall v\in\mathcal{V}
\end{align}

\noindent Notice that some Riemann Solvers relies on the additional parameters related to vehicle turning ratio or driving priority, such as the simple merge and diverge junctions introduced in Section \ref{sectwors}. The vehicle turning ratios may be specified as exogenous parameters obtained from, for example, historical data \citep{CTM1}; they can be also determined endogenously within a dynamic network loading (DNL) model where drivers' route choices are given by the path departure rates \citep{FHNMY}. In the case of the latter, the determination of the vehicle turning ratios must be accompanied by the calculation of link travel times and guided by the {\it first-in-first-out} (FIFO) principle. The reader is referred to \cite{LWRcont} for a complete formulation of the DNL model.

\subsection{Relationship with the link transmission model}\label{subsecLTM}

\cite{LTM} propose the discrete-time  link transmission model (LTM) based on Newell's famous trilogy \citep{Newella, Newellb, Newellc}. In this section, we show the relevance of the LTM with the proposed continuous-time model. 

The LTM keeps track of link-specific variables (e.g. inflow, outflow, demand, and supply) at each time stamp $t^k$, $k=1,\,2\,\ldots$, with a step size denoted by $\Delta t$. It defines the maximum amount of vehicles (volume, not flow) that can be sent by a link $I_i$ during time interval $[t^k,\,t^k+\Delta t]$ to be
\begin{equation}\label{Sbdydiscrete}
\bar S_{i}(t^k)~\doteq~N_{up,i}\left(t^k+\Delta t-{L_i\over v_i}\right)- N_{down,i}(t^k)
\end{equation} 
where $N_{up,i}(\cdot)$ and $N_{down,i}(\cdot)$ denote the cumulative entering and exiting vehicle counts respectively. Notice that $\mathbb{S}_{i}(t^k)$ represents vehicle volume, not flow \footnote{Flow is the rate of change of volume. Their units are respectively vehicle per unit time, and vehicle.}, because of the discrete-time nature. The maximum amount of vehicles (in volume) that can be transmitted from $I_i$ during $[t^k,\, t^k+\Delta t]$ is bounded by
\begin{equation}\label{Slinkdiscrete}
\hat S_{i}(t^k)~=~C_i \cdot\Delta t
\end{equation}
where $C_i$ denotes the link flow capacity. As a result, the amount of vehicles (in volume) that can leave $I_i$ during  $[t^k,\,t^k+\Delta t]$ is the minimum of the two:
\begin{equation}\label{Sidiscrete}
\mathbb{S}_i(t^k)~=~\min\left\{\bar S_{i}(t^k),~\hat S_{i}(t^k)\right\}
\end{equation}
Similarly, for the receiving side of the link model, we have:
\begin{equation}\label{LTMreceiving}
\begin{array}{l}
\bar R_{i}(t^k)~=~N_{down,i}\left(t^k+\Delta t-{L_i\over w_i}\right)+\rho^{jam}_iL_i-N_{up,i}(t^k)
\\
\hat R_{i}(t^k)~=~C_i\cdot\Delta t
\\
\mathbb{R}_{i}(t^k)~=~\min\left\{\bar R_{i}(t^k),~\hat R_{i}(t^k)\right\}
\end{array}
\end{equation}
where $\rho^{jam}_i$ and $L_i$ denotes the jam density and length of link $I_i$, respectively. $w_i$ denotes the speed of backward kinematic waves. 

In the above formulation, the sending capacity $\mathbb{S}_i(t^k)$ and the receiving capacity $\mathbb{R}_i(t^k)$ both represent  traffic volume (in vehicle) in a single time interval; they are different from the demand and supply functions defined in \eqref{demanddef}-\eqref{supplydef}, which represent flow (in vehicle per unit time). Dividing both quantities by $\Delta t$, we get
\begin{align}
{\mathbb{S}_i(t^k)\over \Delta t}~=~&\min\left\{N_{up,i}\left(t^k+\Delta t-{L_i\over v_i}\right)- N_{down,i}(t^k)~,~C_i\right\}  \nonumber
\\
~=~& \min\left\{{N_{up,i}\left(t^k+\Delta t-{L_i\over v_i}\right) -N_{up,i}\left(t^k-{L_i\over v_i}\right) +    N_{up,i}\left(t^k-{L_i\over v_i}\right)  - N_{down,i}(t^k)\over \Delta t}~,~C_i\right\} \nonumber
\\
~=~&\min\left\{f_{in,i}\left(t^k-{L_i\over v_i}\right)+  {  N_{up,i}\left(t^k-{L_i\over v_i}\right)  - N_{down,i}(t^k)\over \Delta t}~,~C_i  \right\}  \nonumber
\\
~=~&
\begin{cases}
f_{in,i}\left(t^k-{L_i\over v_i}\right)\qquad & \hbox{if }~   N_{up,i}\left(t^k-{L_i\over v_i}\right) ~=~ N_{down,i}(t^k)
\\
C_i \qquad  & \hbox{if }~ N_{up,i}\left(t^k-{L_i\over v_i}\right) ~>~ N_{down,i}(t^k)
\end{cases}\qquad \hbox{as }~\Delta t \to 0
\end{align}
and 
\begin{align}
{\mathbb{R}_i(t^k)\over \Delta t}
~=~&\min \left\{  {N_{down,i}\left(t^k+\Delta t-{L_i\over w_i}\right)+\rho^{jam}_iL_i-N_{up,i}(t^k) \over \Delta t}~,~C_i\right\}  \nonumber
\\
~=~&\min \left\{ {N_{down,i}\left(t^k+\Delta t-{L_i\over w_i}\right)    
- N_{down,i}\left(t^k-{L_i\over w_i}\right) \over \Delta t} \right.   \nonumber
\\
&+\left.{N_{down,i}\left(t^k-{L_i\over w_i}\right)+\rho^{jam}_iL_i-N_{up,i}(t^k)\over \Delta t}~,~C_i\right\}   \nonumber
\\
~=~&
\begin{cases}
f_{out,i}\left(t^k-{L_i\over w_i}\right)\qquad &  \hbox{if }~ N_{down,i}\left(t^k-{L_i\over w_i}\right)+\rho_i^{jam}L_i=N_{up,i}(t^k)
\\
C_i  \quad & \hbox{if }~ N_{down,i}\left(t^k-{L_i\over w_i}\right)+\rho_i^{jam}L_i>N_{up,i}(t^k)
\end{cases} \quad  \hbox{as }~\Delta t \to 0
\end{align}
which are recognized as the continuous-time formulations of the demand and supply functions shown in \eqref{DAE2} and \eqref{DAE3}, respectively. Notice that in deriving the above, we have used the forward discretization scheme, also known as the explicit scheme:
\begin{align*}
f_{in,i}\left(t^k-{L_i\over v_i}\right)~\approx~&{N_{up,i}\left(t^k+\Delta t-{L_i\over v_i}\right) -N_{up,i}\left(t^k-{L_i\over v_i}\right)\over \Delta t}
\\
f_{out,i}\left(t^k-{L_i\over w_i}\right)~\approx~& {N_{down,i}\left(t^k+\Delta t-{L_i\over w_i}\right)    
- N_{down,i}\left(t^k-{L_i\over w_i}\right) \over \Delta t}
\end{align*}
Therefore, the LTM can be regarded as a special case of the proposed DAE system with forward time-discretization scheme.

\begin{remark}
The continuous-time DAE system, although does not facilitate computation  directly unless it is discretized, is valuable for the following reasons. (1) It describes the underlying network model of which various existing discrete models, such as the cell transmission model \cite{CTM1, CTM2} and the link transmission model \cite{LTM}, are approximations. Thus it provides a direct reference for analyzing the numerical behavior of these discrete models and their convergence when the time grid is refined. (2) The continuous-time version admits several discretization schemes, and the LTM is just one of them. In particular, one could introduce backward or central discretization schemes \citep{LeVeque} to \eqref{DAE1} instead of the forward scheme, which gives rise to new models not studied before. 
\end{remark}

\vspace{0.3 in}

\section{Global existence of weak solutions}\label{secexistence}

This section provides an existence theory for  networks consisting of junctions depicted in Figure \ref{figtwojunc}, with Riemann Solvers  described in Section \ref{secdiverge} and Section \ref{secmerge}. As such, the network model assumes that the vehicle turning ratios at diverge junctions are fixed constants.

We show the existence of a network-wide weak solution based on the {\it wave front tracking} (WFT) algorithm  \citep{Dafermos1972, CGP, GPbook}. This algorithm constructs piecewise constant approximations of the weak solution to the scalar conservation law of the form \eqref{LWRPDE}. The WFT method has been used to show existence of solutions of single conservation laws as well as systems \citep{GPbook, HR2002}. This section provides existence result for networks consisting of merge and diverge junctions mentioned in Section \ref{sectwors}.

\subsection{A brief review of the wave front tracking method}\label{subsecwft}

To make our presentation more comprehensive and to be self-contained, this section briefly explains procedures of wave front tracking. Interested readers are referred to \cite{Bbook} and \cite{HR2002} for more details on conservation laws, and to \cite{GPbook} for the traffic network case.

Given a discretization parameter $\delta$ and initial-boundary conditions with {\it bounded variation} (BV), an approximate solution on the network of interest is constructed in the following way.

\begin{itemize}
\item Approximate the initial-boundary conditions by piecewise constant (PWC) functions, and solve the Riemann Problems (RPs) at discontinuities of these PWC functions and at junctions.
Approximate rarefaction waves by rarefaction shocks of size $\delta$.

\item Construct the PWC solution by piecing together the
solutions of the RPs up to the first time when two traveling waves (shocks) interact, or when a wave interacts with a junction.

\item Solve a set of new RPs created by the interactions and prolong the
solution up to next interaction time, and so on.
\end{itemize}

To ensure the feasibility of such a construction and hence the existence of the WFT approximate solution, it suffices to estimate and bound the number of waves and the number of interactions among themselves and with the junctions. This can be easily done in the scalar conservation law case since  the number of waves and the total variation of the approximate solution are both decreasing in time \citep{Bbook}. For the network case, one needs to carefully estimate the number of waves, the number of interactions, and the total variations, which are complicated by the interactions of waves with junctions, as well as ways in which such interactions occur, given the Riemann Solvers. Indeed, when a wave interacts with a junction from a link, it may produce new waves in all other links connected to the same junction; this is in contrast to the single conservation law case. We refer the reader to \cite{GPbook} for more elaboration on these insights.

Now consider a sequence of approximate solutions $\rho_\delta$ constructed
by a wave-front tracking algorithm mentioned above. If one can provide estimates on the total variation of the flow,
then as $\delta$ tends to zero
$\rho_\delta$ tends to a weak entropy solution 
on the whole network.
To this end, one typically employs the compactness in BV provided by Helly's Theorem
and the notion of weak formulation of equation \eqref{LWRPDE} \citep{Bbook,
GPbook}.

\subsection{Existence of a weak solution at a junction}

As we mentioned previously, the estimations of number of weaves and interactions are highly specific to the junction type and the Riemann Solver employed. Thus there is hardly any universal result on the existence of network solutions.  \cite{GP2009} are the first to provide sufficient existence conditions by stipulating a few abstract assumptions on the Riemann Solvers considered for the junctions. Our strategy for showing existence is by proving that the two Riemann Solvers presented in Section \ref{sectwors} satisfy the sufficient conditions given by \cite{GP2009}. We need the following straightforward terminologies.

\begin{definition}{\bf (Good and bad data)}\label{goodbaddef}
Fix an approximate wave front tracking solution $\rho_{\delta}$ and a junction $J$. An incoming link $I_i$ is said to have a good datum at $J$ at time $t$ if
$$
r_i(t,\,b_i-)~=~1,
$$
and it has a bad datum otherwise. Similarly, we say that an outgoing link $I_j$ has a good datum at $J$ at $t$ if
$$
r_j(t,\,a_j+)~=~0,
$$
and it has a bad datum otherwise. Here $r_i(t,\,x)$ and $r_j(t,\,x)$ are both binary variables indicating the free-flow/congested phase. 
\end{definition}

In \cite{GP2009}, the three sufficient conditions to be satisfied by a Riemann Solver to ensure the existence of a network solution are specified as follows.  
In rough words such conditions can be expressed as follows:
\begin{itemize}
\item[(P1)] The solution of a Riemann Problem (RP) depends only on the values of bad data.

\item[(P2)] For waves interacting with a junction, the change in the flow variation (i.e. the sum of jumps in flow over all waves) due to the interaction is bounded, up to a multiplication by a constant, by the change in flow through the junction and by the flow jump of the interacting wave.

\item[(P3)] Interactions of waves bringing a flow decrease produce a decrease in the flow through the junction.
\end{itemize}

\noindent Let us now establish some precise mathematical interpretations of the three properties (P1)-(P3).\\

\noindent {\bf [Interpretation of P1]} Recalling Definition \ref{chapDNL:lwrrsdef}, 
we say that a Riemann solver $RS$ satisfies property (P1) if 
$$RS(\hat \rho_1,\ldots,\hat\rho_{m+n})=RS(\tilde\rho_1,\ldots,\tilde \rho_{n+m})$$
whenever, for every $i=1,\ldots,n+m$, either $\hat\rho_i=\tilde\rho_i$ or both $\hat\rho_i$
and $\tilde\rho_i$ are good data. In particular, only the values of bad data matters.\\

\noindent {\bf [Interpretation of P2]} Let us first define the flow through the junction at time $t$ as:
\begin{equation}\label{eq:Gamma}
\Gamma(t)=\sum_{i=1}^m f_i(\rho_i(t,b_i-),
\end{equation}
and consider an interacting wave, say from an incoming road $I_i$.  We use the superscripts $-$ and $+$ to indicate, respectively, variables before and after the interaction. Moreover,  
let $(\rho_i,\rho_i^-$) be the interacting wave \footnote{We use the notation $(\rho_i,\,\rho_i^-)$ to represent a wave with state $\rho_i$ behind the wave and state $\rho_i^-$ in front of the wave.}. 
The change in the {\it total variation} (TV) of the flow, due to the interaction, is given by:
\begin{equation}\label{TVdef}
\Delta TV(f)=\sum_{k\not= i}|f_k(\rho^+_k)-f_k(\rho^-_k)|
+|f_i(\rho_i)-f_i(\rho^+_i)| - |f_i(\rho_i)-f_i(\rho^-_i)|    
\end{equation}
Indeed, before the interaction we have the wave $(\rho_i,\rho_i^-)$ on road $I_i$. After the interaction
we have the wave $(\rho_i,\rho_i^+)$ on road $I_i$ and, for $k\not= i$, the wave
$(\rho^+_k,\rho^-_k)$ on road $I_k$.
Then (P2) holds if there exists $C>0$ (independent of the interacting wave) such that:
\begin{equation}
|\Delta TV(f)|\leq C\min\left\{\left|\Gamma(t+)-\Gamma(t-)\right|, |f_i(\rho_i)-f_i(\rho^-_i)| \right\} \qquad\forall i
\end{equation}

\noindent {\bf [Interpretation of P3]} To interpret (P3), consider again an interacting wave from any road.
Using the same notations for (P2),  (P3) requires that if $f(\rho_i)<f(\rho_i^-)$
then $\Gamma(t+)\leq\Gamma(t-)$.

\subsection{Existence result} 

\begin{lemma}\label{existencediverge}
The Riemann Solver for the diverge junction introduced in  Section \ref{secdiverge} satisfies properties (P1)-(P3), thus a weak solution of the Cauchy problem at such a junction exists. 
\end{lemma}
\begin{proof}
The proof is postponed until Appendix \ref{secapplemmaexistencediverge} to improve the readability of this paper.
\end{proof}

\begin{lemma}\label{existencemerge}
The Riemann Solver for the merge junction in Section \ref{sectwors} satisfies properties (P1)-(P3), thus a weak solution of the Cauchy problem at the junction exists. 
\end{lemma}
\begin{proof}
For the same reason as before, we will present the proof in Appendix \ref{secappexistencemerge} instead.
\end{proof}

\begin{theorem}{\bf (Global existence result for the initial value problem on a network)} 
For a network consisting of diverge and merge junctions whose Riemann Solvers are given by 
\eqref{fifodiv} and \eqref{RSmergedef} respectively, a weak solution on this network exists.
\end{theorem}
\begin{proof}
In order to extend the local existence results conveyed by Lemma \ref{existencediverge} and \ref{existencemerge} to a network, we note the fact that in a network described by a system of conservation laws, the propagation of any condition or perturbation has a finite speed, bounded by the speed of the kinematic waves.  Thus the local existence results lead to a global existence result. 
\end{proof}

\section{Well-posedness of initial value problems}\label{secwellposedness}

In the mathematical modeling of a physical system, the term {\it well-posedness} refers to the property of having a unique solution, and the behavior of that solution hardly changes when there is a slight change in the initial/boundary condition. Examples of well-posed problems include the Cauchy problem (initial value problem) for scalar conservation laws of the form \eqref{LWRPDE} \citep{Bbook}, and the Cauchy problem for Hamilton-Jacobi equations obtained by integrating \eqref{LWRPDE} \citep{VT2}.

\subsection{An introduction to generalized tangent vectors}\label{secgtv}

To prove  the well-posedness for the Cauchy problem on a junction where the dynamic on each adjacent link is governed by the scalar conservation law \eqref{LWRPDE}, 
one has several options such as Kruzhov's entropy conditions  \citep{Dafermos},
the Bressan-Liu-Yang functional \citep{Bbook}, and the generalized tangent vectors \citep{GPbook}. For our specific problem the generalized tangent vector approach turns out to be the most convenient. This approach will be explained below. 

Given a piecewise constant function $F(x) :[a,\,b]\to \mathbb{R}$, a {\it tangent vector} is
defined in terms of the shifts of the discontinuities of $f(\cdot)$.
More precisely, let us indicate by $\{x_i\}_{i=1}^N$ the discontinuities of $F$ where
$a=x_0< x_1 < \cdots <x_{N}< x_{N+1}=b$, and
by $\{F_i\}_{i=1}^N$ the values of $F$ on $(x_{i-1},\,x_i)$.
 A tangent vector of $F(\cdot)$ is a vector $\xi=(\xi_1,\ldots,\xi_N)\in \mathbb{R}^N$ such that 
for each $\epsilon>0$, one may define the corresponding perturbation of $F(\cdot)$, denoted by $F^{\epsilon}(\cdot)$ and given by
\[
F^{\epsilon}(x)=F_i\qquad x\in[x_{i-1}+\epsilon\xi_{i-1}, \, x_i+\epsilon\xi_i)
\]
for $i=1,\ldots,N+1$, where we set $\xi_0=\xi_{N+1}=0$. The norm of the tangent vector $\xi$ is defined as
\begin{equation}\label{tvnormdef}
\|\xi\|~\doteq~\sum_{i=1}^N |\xi_i| \cdot  |F(x_i+)-F(x_i-)|
\end{equation}
In other words, the norm of the tangent vector is the sum of the magnitude of each  $\xi_i$ multiplied by the size of the shifted jumps.

In order to show well-posedness of a Cauchy problem at a road junction, one proceeds as follows.
Given piecewise constant initial-boundary conditions on each link incident to that junction, one considers their tangent vectors. By showing that the norms of their tangent vectors
are uniformly bounded in time among approximate wave-front tracking solutions, one guarantees that  the $L^1$ distance of any two solutions is bounded, up to a constant, by the $L^1$ distance of their respective initial conditions. More precisely, we have the following theorem.

\begin{theorem}
If the norm of a tangent vector does not increase in time among all approximate wave-front tracking solutions, then the Cauchy problem on a junction is well-posed and has a unique solution.
\end{theorem}
 \begin{proof}
 See \cite{GPbook}.
 \end{proof}

 \subsection{Well-posedness of the junction models}
A key fact necessary for the success of the above procedure is that when a wave interacts with the junction, the norm of the tangent vector changes in the same way as the jumps in the flows. 
To make our statement precise, let us consider a junction $J$ and assume the interacting wave $(\rho_i,\,\rho_i^-)$  is coming from road $I_i$ and generates a new wave $(\rho_j^-,\,\rho_j^+)$ on some other link $I_j$ after interacting with the junction.  We indicate by  $\Delta q_i$ and $\Delta q_j$ the flow jumps in these waves; that is,
$$
\Delta q_i~=~f_i(\rho_i^-)-f_i(\rho_i),\qquad\qquad \Delta q_j~=~f_j(\rho_j^+)-f_j(\rho_j^-)
$$
 In addition, let us imagine a shift of the wave $(\rho_i,\,\rho_i^-)$ by $\xi_i$. Such a shift will certainly cause a shift of the wave $(\rho_j^-,\,\rho_j^+)$, whose amount is denoted by $\xi_j$.  We have the following lemma which relates all the quantities mentioned above in one identity. This result  is due to \cite{GPbook}.
 \begin{lemma}
If the wave on $I_i$ interacts with $J$ without producing waves in the same road $I_i$, then, the shift $\xi_j$ produced on $I_j$, as a result of the shift $\xi_i$ on $I_i$,  satisfies
\begin{equation}\label{eq:tangent-vectors}
\xi_j\big(\rho_j^+-\rho_j^-\big)~=~{\Delta q_j\over \Delta q_i}\,\xi_i \left(\rho_i^--\rho_i\right)
\end{equation}
where $\rho^-_j,\,\rho^+_j$ are the states at $J$ on road $I_j$ before and after the interaction, respectively.
\end{lemma}
\begin{proof}
The proof is given by \cite{GPbook}
\end{proof}

Notice that the quantity $\left|\xi_i\left( \rho_i^- - \rho_i\right)\right|$ is precisely the $L^1$-distance of two initial conditions on $I_i$ with and without the shift $\xi_i$ at the discontinuity $(\rho_i,\,\rho_i^-)$, respectively. Similarly, $\left|\xi_j\left( \rho_j^+ - \rho_j^-\right)\right|$ is the $L^1$-distance of the two conditions on $I_j$ as a result of having or not having the initial shift $\xi_i$,  respectively. To make things even more explicit, if $\xi_i$ is the only shift in the initial condition on $I_i$, then $\left|\xi_i\left( \rho_i^- - \rho_i\right)\right|$ is the norm of the tangent vector for $I_i$ (see definition \eqref{tvnormdef}) before the interaction, and $|\xi_j( \rho_j^+ - \rho_j^-)|$ is the norm of the tangent for $I_j$ after the interaction

From \eqref{eq:tangent-vectors}, we have
\begin{equation}\label{tangent-vectors}
\left|\xi_j \big(\rho_j^+-\rho_j^-\big)\right|~=~{|\Delta q_j|\over|\Delta q_i|}  \left|\xi_i \left(\rho_i^- -\rho_i \right)\right|
\end{equation}
Therefore, to bound the norm of the tangent vectors one has to check that the multiplication
 factors $\frac{|\Delta q_j|}{|\Delta q_i|}$ remains uniformly bounded, regardless of  the number of interactions that may occur at this junction.

We are ready to prove the following:
\begin{theorem}
For both the merge and diverge junctions, the norms of the tangent vectors shown as
\eqref{tangent-vectors} ares uniformly bounded.
\end{theorem}

\begin{proof}

\noindent {\bf (Part 1.)} Let us start with the diverge junction.
Consider a wave interacting from road $I_i$ producing a wave on road $I_j$,
where $i,j=1,2,3$. By (\ref{fifodiv}), the change in the boundary flows after the interaction, denoted by $\Delta q_{1},\,\Delta q_{2}$ and $\Delta q_{3}$ satisfy 
$$
\Delta q_{2}~=~\alpha_{1,2} \,\Delta q_{1},\qquad \Delta q_{3}~=~\alpha_{1,3}\,\Delta q_{1}
$$
Consequently, the multiplication factors (see \eqref{tangent-vectors}) satisfy 
$$
\begin{array}{lll}
\displaystyle {|\Delta q_1|\over |\Delta q_1|}~=~1, \qquad  & \displaystyle {|\Delta q_2|\over |\Delta q_1|}~=~\alpha_{1,2},  \qquad &\displaystyle {|\Delta q_3|\over |\Delta q_1|}~=~ \alpha_{1,3}; \\\\
\displaystyle {|\Delta q_1| \over |\Delta q_2|}~=~{1\over \alpha_{1,2}},   \qquad & \displaystyle {|\Delta q_2|\over  |\Delta q_2|}~=~1, \qquad & \displaystyle {|\Delta q_3| \over |\Delta q_2|}~=~{\alpha_{1,3}\over \alpha_{1,2}}; \\\\
\displaystyle {|\Delta q_1| \over |\Delta q_3|}~=~{1\over \alpha_{1,3}}, \qquad   & \displaystyle {|\Delta q_2|\over |\Delta q_3|}~=~{\alpha_{1,2}\over\alpha_{1,3}}, \qquad  &\displaystyle {|\Delta q_3|\over |\Delta q_3|}~=~1.
\end{array}
$$
That is, the multiplication factors of the tangent vector norms, denoted by $q_{ij}$,  
are given by the following table:
\begin{equation}
\begin{array}{ccc}
1 &  \alpha_{1,2} & \alpha_{1,3} \\
\frac{1}{\alpha_{1,2}} & 1 & \frac{\alpha_{1,3}}{\alpha_{1,2}}\\
\frac{1}{\alpha_{1,3}} &  \frac{\alpha_{1,2}}{\alpha_{1,3}} & 1
\end{array}
\end{equation}
Notice that
\[
q_{ij}q_{jk}=q_{ik},\qquad \forall  i,\,j,\,k~=~1,\,2,\,3
\]
this means that no matter how many interactions occur, the multiplication factor, which is a product of individual $q_{ij}$'s, is always bounded uniformly. Therefore, the norms of the tangent vectors are uniformly bounded, independent of the number of interactions at the junction.\\

\noindent {\bf (Part 2.)} Let us now turn to the merge junction. Similarly, we need to check the 
multiplication factors $q_{ij}$ of the norms of the tangent vectors. Assume first that a wave interacts from road $I_6$, then in the case depicted in the left part  of Figure
\ref{figmerge} we have $q_{6,4}=p$, $q_{6,5}=1-p$ and $q_{6,6}=1$.
Otherwise, if the right part of Figure \ref{figmerge} occurs, then either
$q_{6,4}=1$, $q_{6,5}=0$ and $q_{6,6}=1$, or
$q_{6,4}=0$, $q_{6,5}=1$ and $q_{6,6}=1$. The multiplication factors are certainly uniformly bounded. \\
Assume now that a wave interacts from road $I_4$, then we distinguish between two cases:
\begin{itemize}
\item[(i)] $D_4(\rho_4^-)+D_5(\rho_5^-)<S_6(\rho_6^-)$;
\item[(ii)] $D_4(\rho_4^-)+D_5(\rho_5^-)\geq S_6(\rho_6^-)$.
\end{itemize}
If case (i) occurs, then $q_{4,4}=1$, $q_{4,5}=0$ and $q_{4,6}=1$.
In case (ii) we have $q_{4,4}=1$, $q_{4,5}=\frac{1-p}{p}$ and $q_{4,6}=0$. \\
The case where a wave from road $I_5$ interacts with the junction is entirely similar and we get the following.
If case (i) occurs, then $q_{5,4}=0$, $q_{5,5}=1$ and $q_{5,6}=1$.
If case (ii) occurs, then $q_{5,5}=1$,
$q_{5,4}=\frac{p}{1-p}$ and $q_{5,6}=0$. Since $q_{ij}q_{jk}\leq q_{ik}$ for $i,\,j\,k=4,\,5,\,6$, we reach the same conclusion about uniform boundedness for the diverge junction. 
\end{proof}

\section{Numerical examples}\label{secNumerical}

\subsection{Visualization of spillback}
In this section, we provide an illustrate of shock waves and vehicle spillback on a traffic network, based on the computational procedure described in Section \ref{secdiscreteDAE}. The simple seven-link network shown in Figure \ref{figseven} is considered where triangular fundamental diagrams are employed for all the links, and link-specific parameters are summarized in Table \ref{tablinkpara}. For the merge and diverge junction models, we apply the Riemann Solvers discussed in Section \ref{sectwors}, where the vehicle turning ratios and the right-of-way parameters are set to be 1/2.

\begin{figure}[h!]
\centering
\includegraphics[width=0.45\textwidth]{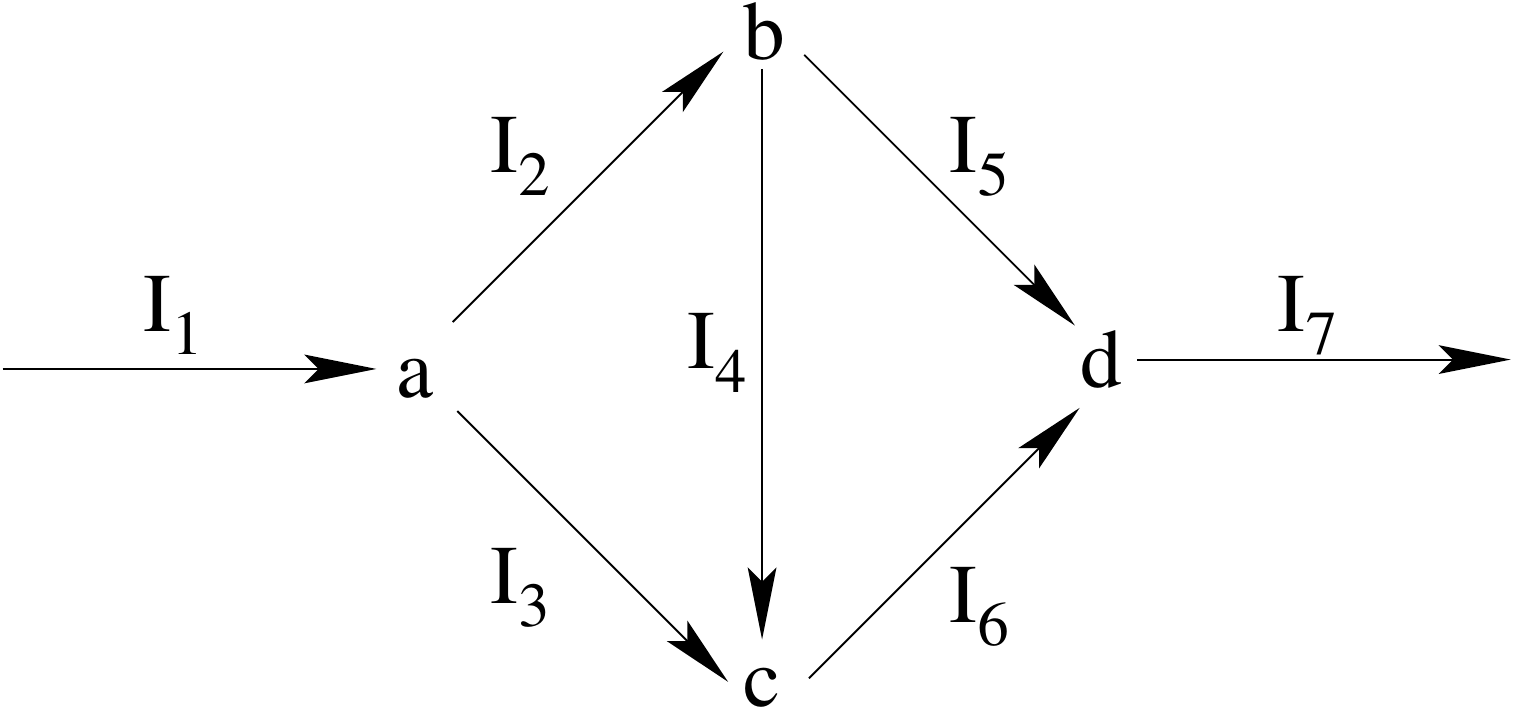}
\caption{A network example consisting of seven links and four nodes}
\label{figseven}
\end{figure}

\begin{table}[h!]
\begin{center}
\begin{tabular}{|c|c|c|c|c|c|}
\hline
Arc  & $\rho_i^{jam}$    &  $v_i$     &   $w_i$  &  $C_i$  & $L_i$ \\
     &   (vehicle/mile)     &  (mile/hour)   & (mile/hour)&  (vehicle/hour)  &  (mile)   \\
\hline\hline
$I_1$   &  400 &  30 &  10  &  3000 & 3 \\
\hline
$I_2$ &  200   &  30 & 10   &  1500  &  3\\
\hline 
$I_3$ &   400  &  30  & 10  &  3000  &  3 \\
\hline
$I_4$ &  100   & 30 &  10    &   750   &  3\\
\hline
$I_5$  &  200  &  30 & 10    & 1500   & 3\\
\hline 
$I_6$  & 200   &  30  & 10    &  1500   &  3 \\
\hline
$I_7$ &  200   & 30  & 10     &  1500   &  3 \\
\hline
\end{tabular}
\caption{\small Link parameters for the small network. $\rho_i^{jam}$: jam density; $v_i$: forward wave speed; $w_i$: backward wave speed; $C_i$: flow capacity; $L_i$: length.}
\label{tablinkpara}
\end{center}
\end{table}

We consider a time horizon of $[0,\,5]$ (in hour), with a time step $\delta t=0.05$ hour (3 minutes).  The time-dependent departing flow into the network is randomly generated between $0$ and the flow capacity $C_1$. The departure window is set to be $[0.5,\,3.5]$ (in hour).

As indicated by Table \ref{tablinkpara}, link $I_4$ presents a potential bottleneck at node $b$ as it has a very low flow capacity. In addition, a potential bottleneck also exists at the  merge node $c$  since the capacity of the downstream link $I_6$ is less than the sum of flow capacities of $I_3$ and $I_4$.  In order to have a clear visualization of the propagation of kinematic waves and the separating shock wave on each link, we use the boundary data $ q_{out,i},\, q_{in, i},\,i=1,\ldots,  7$  to construct the Moskowtiz functions through the Lax-Hopf formula \eqref{simpleLH}. For the Moskowitz surface, the separating shock is no longer represented as a discontinuity in the function; rather, it is shown as a `kink'  (discontinuity in the first derivative). The Moskowitz functions on links $I_1, I_2, I_3$ and $I_4$ are shown in Figures \ref{figlink1}, \ref{figlink2}, \ref{figlink3} and \ref{figlink4}, respectively.

Two regions are observed from the 3-D plots of the Moskowitz functions, which represent respectively the free-flow phase and the congested phase of traffic. The mutual boundary of these two regions is identified as the separating shock. One can observe from Figures \ref{figlink3} and \ref{figlink4} that the congested regions on links $I_3$ and $I_4$ initiate from their downstream boundaries (node $c$) and propagates backward towards their entrances and then further to their respective upstream links $I_1$ and $I_2$.  Such spillback phenomena are visualized in Figures \ref{figspillbackvis}, where we place relevant Moskowitz functions side by side and view them from above. It is clearly seen that the congested regions `penetrate' the boundaries between the two links, which is a clear indication of the propagation of congestion among different links of the network. We also see that the separating shocks begin to retreat around time $t=3.5$, when the departing flow becomes zero, indicating that the queues are dissipated as a result of substantially reduced inflow into the network.

\begin{figure}[h!]
\begin{minipage}[b]{.49\textwidth}
\centering
\includegraphics[width=\textwidth]{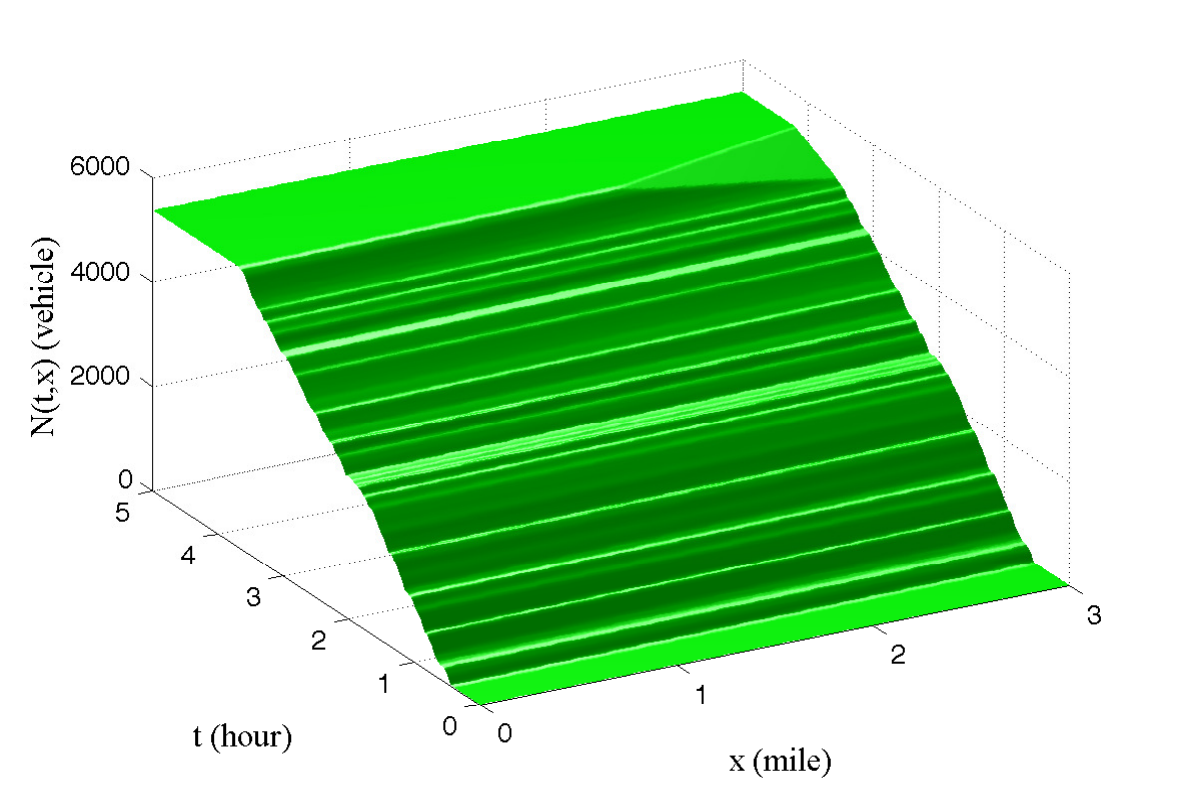}
\caption{Moskowitz function for link $I_1$.}
\label{figlink1}
\end{minipage}
\begin{minipage}[b]{.49\textwidth}
\centering
\includegraphics[width=\textwidth]{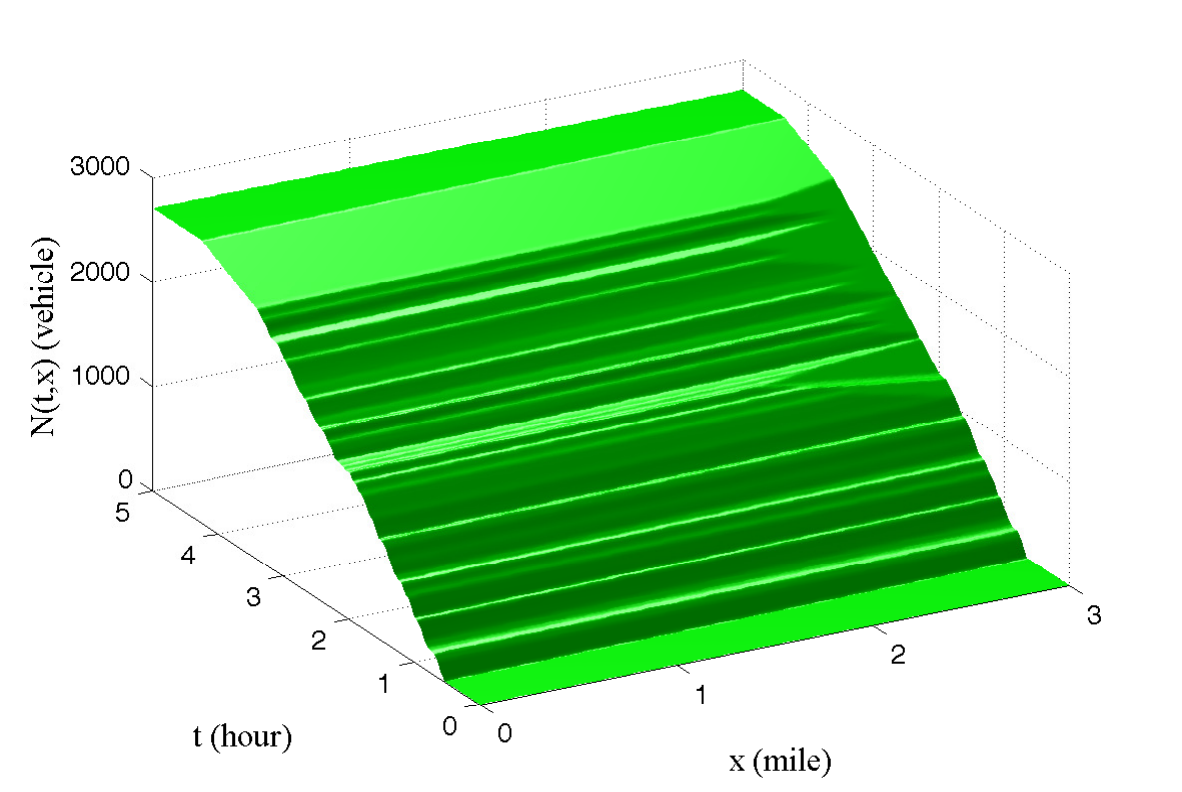}
\caption{Moskowitz function for link $I_2$.}
\label{figlink2}
\end{minipage}
\end{figure}

\begin{figure}[h!]
\begin{minipage}[b]{.49\textwidth}
\centering
\includegraphics[width=\textwidth]{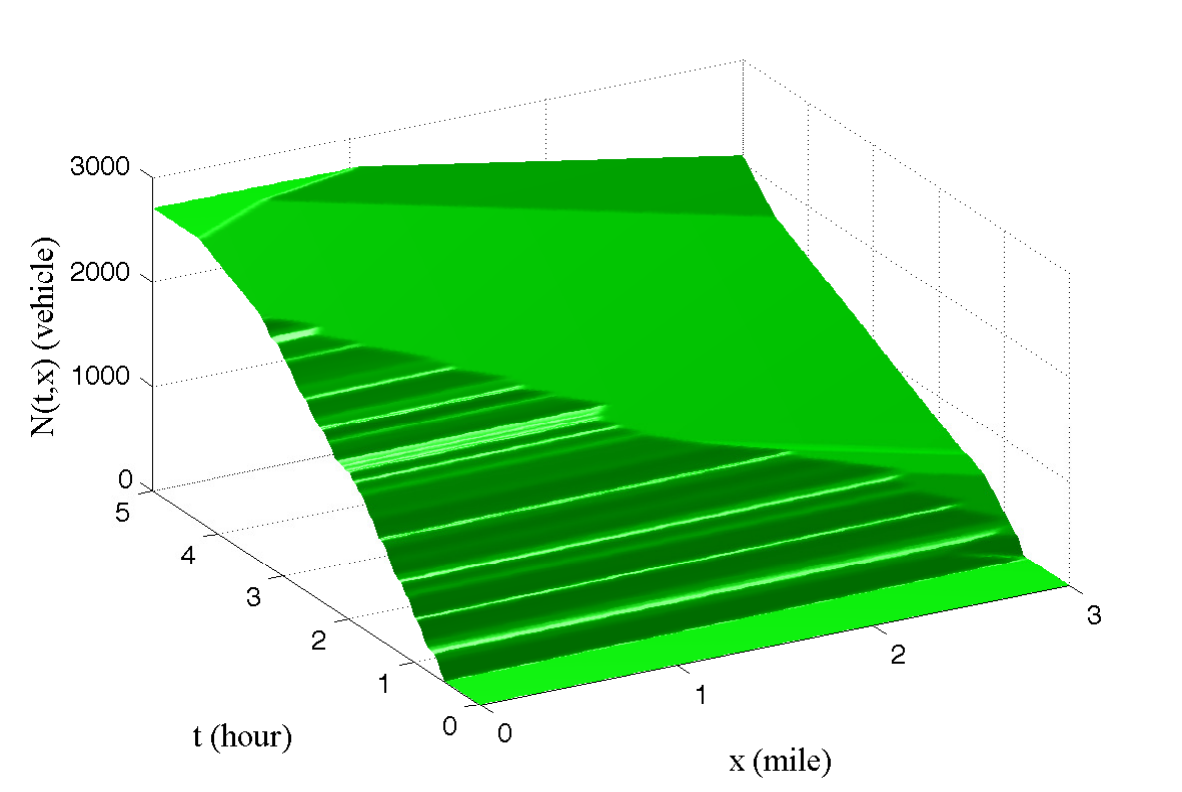}
\caption{Moskowitz function for link $I_3$.}
\label{figlink3}
\end{minipage}
\begin{minipage}[b]{.49\textwidth}
\centering
\includegraphics[width=\textwidth]{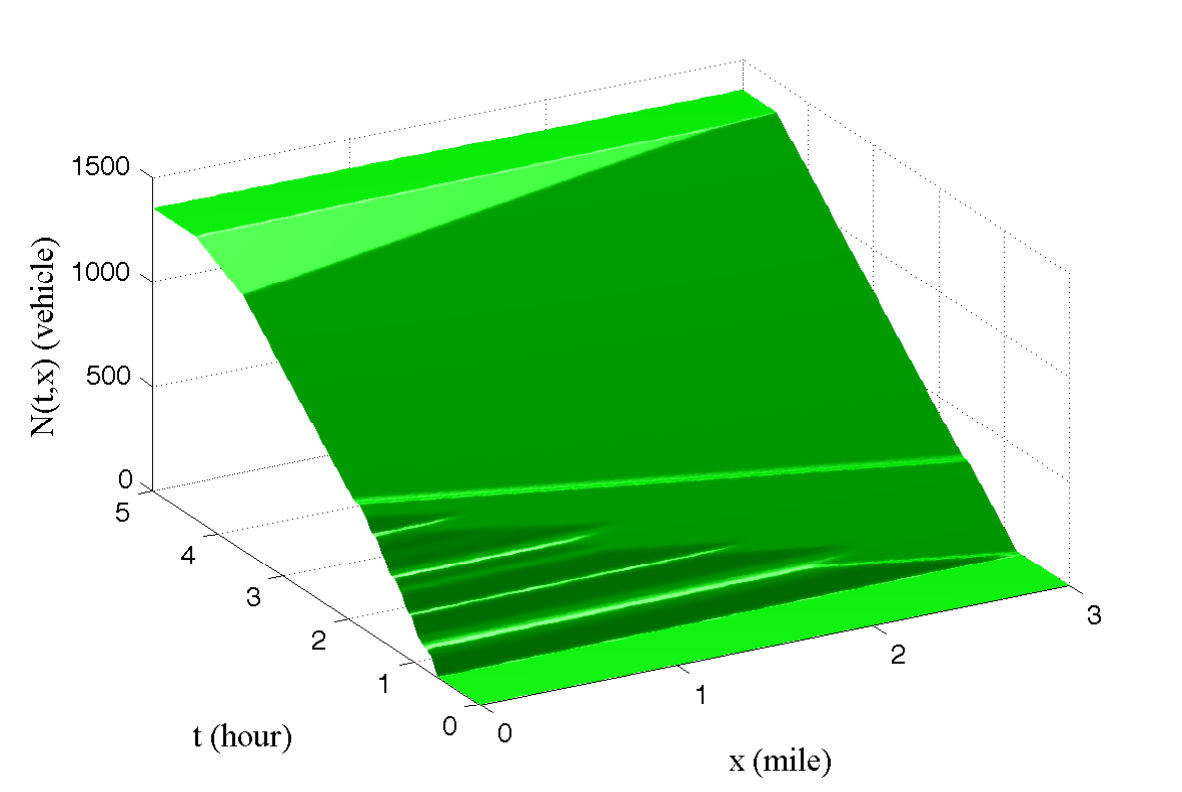}
\caption{Moskowitz function for link $I_4$.}
\label{figlink4}
\end{minipage}
\end{figure}

\begin{figure}[h!]
\begin{minipage}[b]{.49\textwidth}
\centering
\includegraphics[width=\textwidth]{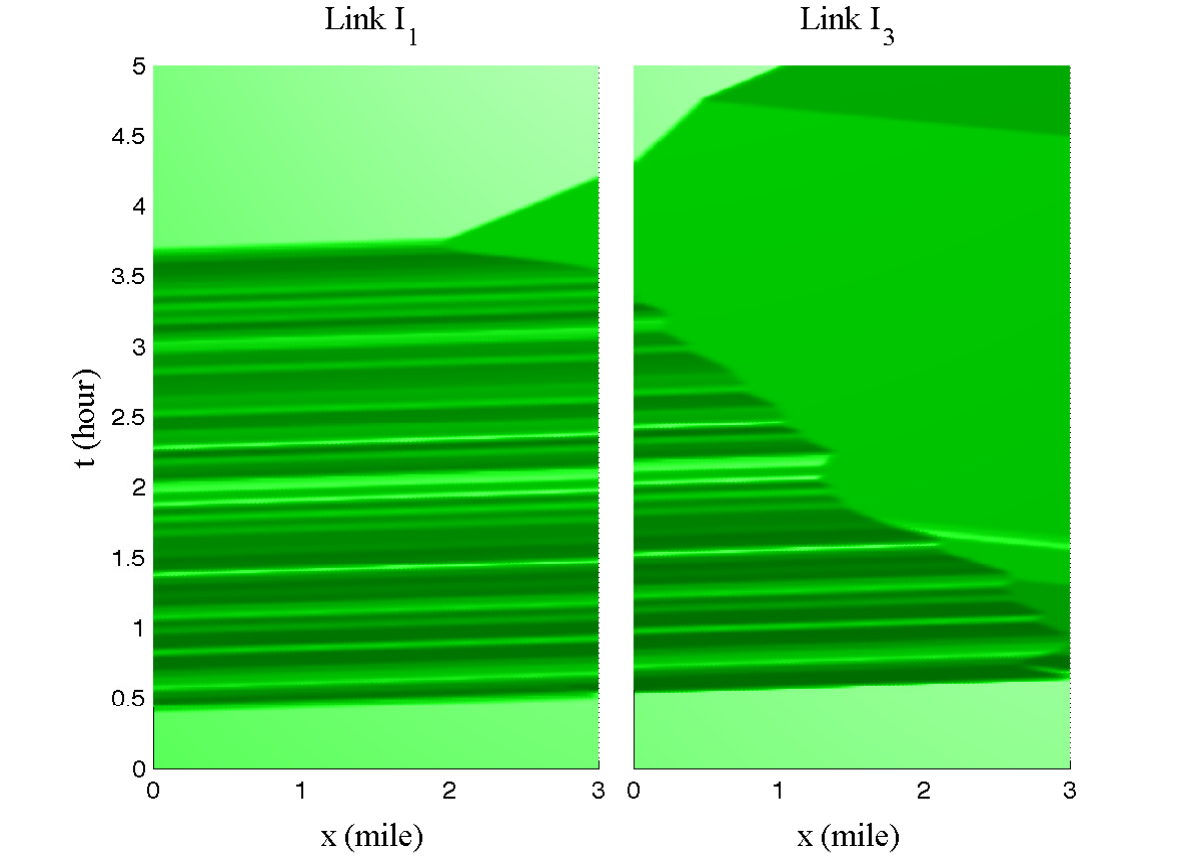}
\end{minipage}
\begin{minipage}[b]{.49\textwidth}
\centering
\includegraphics[width=\textwidth]{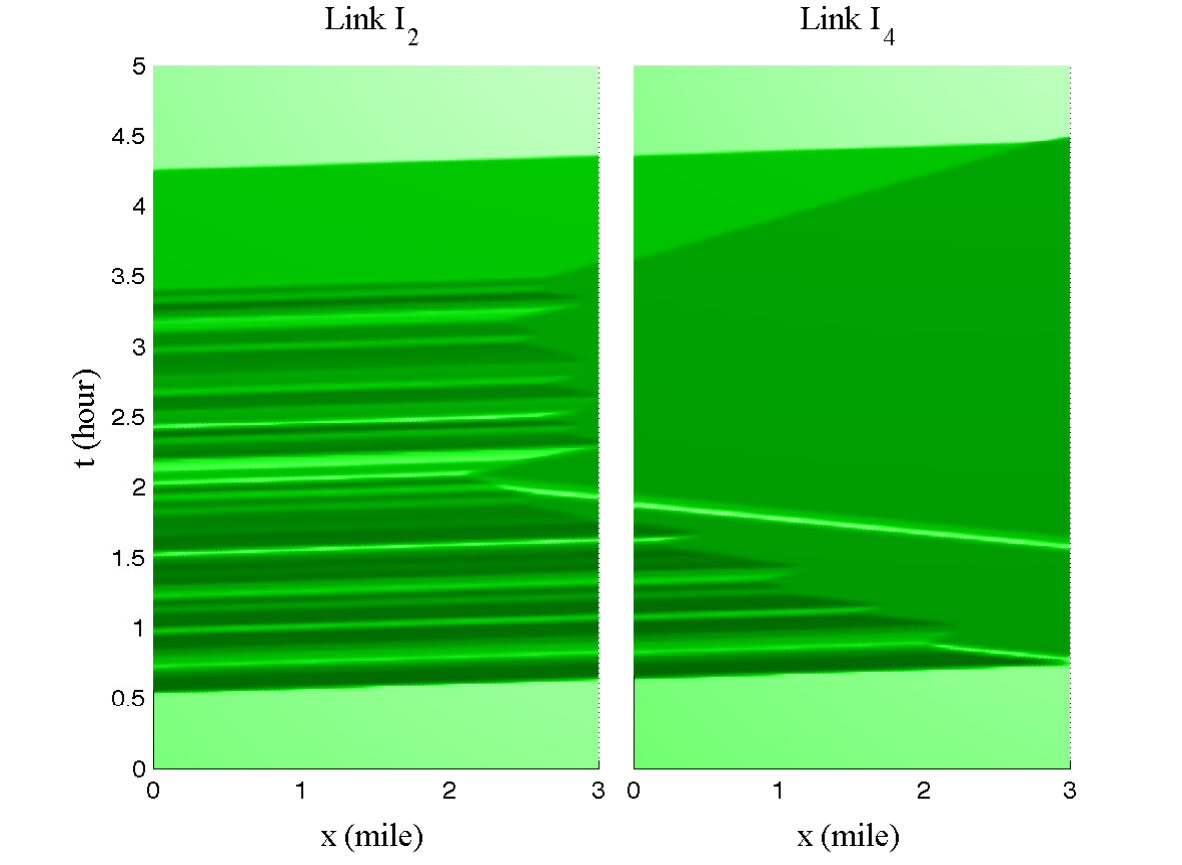}
\end{minipage}
\caption{Moskowitz functions of links $I_1$ and $I_3$ (left), and $I_2$ and $I_4$ (right). The congestion on link $I_3$ ($I_4$) spills over into link $I_1$ ($I_2$).}
\label{figspillbackvis}
\end{figure}

\subsection{Computational time}

This section assesses the computational performance of the discretized DAE system on the seven-link network as well as a few larger traffic networks, which are depicted in Figure \ref{fignetworks}. The network data are provided in http://www.bgu.ac.il/~bargera/tntp/. All computations reported in this section were coded in MATLAB and performed on a standard laptop with 8GB RAM.

\begin{figure}[h!]
\centering
\includegraphics[width=\textwidth]{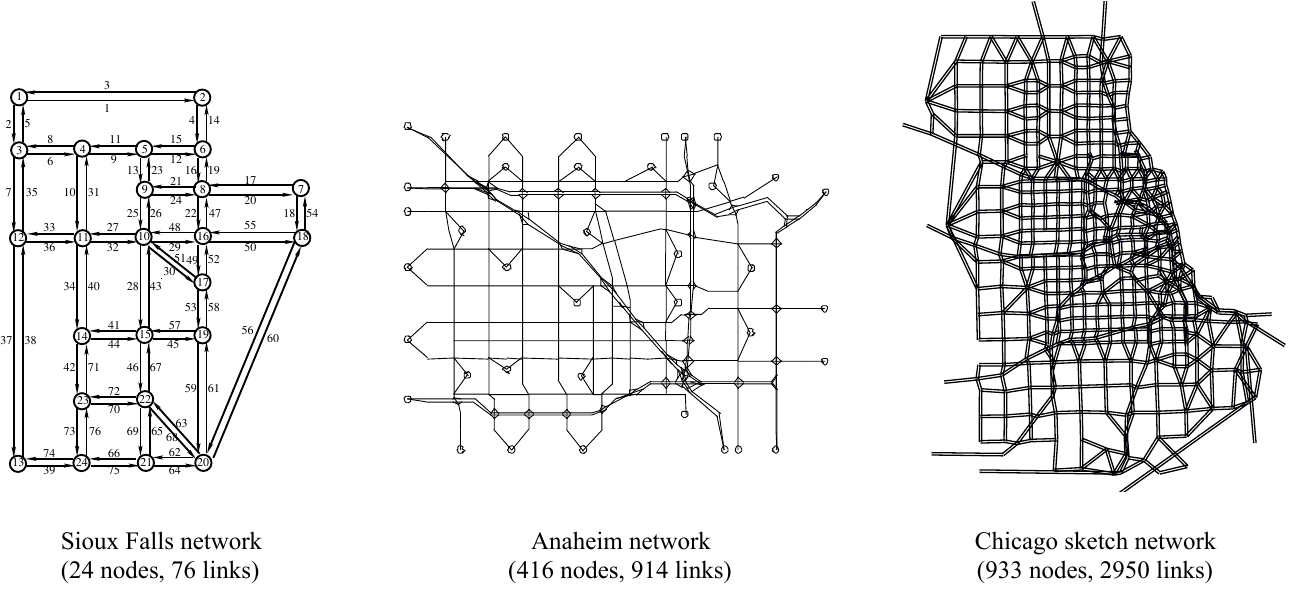}
\caption{Test networks.}
\label{fignetworks}
\end{figure}

For each network we solve the discretized DAE system with a given number of time intervals, denoted $N$. Since these networks contain junctions of more complex topologies than what is described in Section \ref{sectwors}, we consider a junction model proposed by \cite{HGPFY}, which aims to approximate the on-and-off signal timing that exists at a road intersection. An important parameter of this continuum signal junction model is $\eta_i$, which represents the green time ratio within a full signal cycle that is allocated to each incoming link $I_i$ of the intersection. And, such $\eta_i$'s typically need to satisfy $\sum_{I_i\in\mathcal{I}^J} \eta_i\leq1$ to ensure that no conflicting traffic streams are discharged during a signal phase. We refer the reader to \cite{HGPFY} for more discussion on this model and its implications to the modeling of signalized traffic networks.

Each test network is simulated with random inflows at origin nodes. The vehicle turning percentages and green ratio $\eta_i$ are fixed constants in our computations. To sufficiently test the computational time, we consider four values of $N$, the number of time intervals: 100, 200, 400, and 800. Table \ref{tabcputime} presents the computational time for each scenario. We can see from the table (and also the left picture of Figure \ref{figCPUtime}) that the computational times for each test network grow linearly with respect to the number of time steps $N$; that is, the computational time is on the order of $O(N)$. This property of link-based models is in contrast with cell-based models, in which the memory usage and number of flops (floating point operations) are both on the order of $O(N^2)$, due to the spatial discretization and the Courant-Friedrichs-Lewy condition \citep{LeVeque}. This highlights the computational efficiency of the proposed DAE system and link-based approaches in general.

\begin{table}[h!]
\begin{center}
\begin{tabular}{|c|c|c|c|c|}
\hline
               & Seven-arc    & Sioux Falls   &  Anaheim     &   Chicago sketch \\
\hline\hline
$N=100$   &  0.08 s &  0.16 s &  1.30 s  &  1.83 s \\
\hline
$N=200$ &   0.19 s  &   0.30 s &  2.65 s   & 3.52 s\\
\hline 
$N=400$ &   0.41 s  &    0.59 s  & 5.13 s   & 7.16 s\\
\hline
$N=800$ &   0.83 s   &   1.39 s &   10.39 s   & 13.98 s\\
\hline
\end{tabular}
\caption{Summary of computational times.}
\label{tabcputime}
\end{center}
\end{table}

\begin{figure}[h!]
\begin{minipage}[b]{.49\textwidth}
\centering
\includegraphics[width=\textwidth]{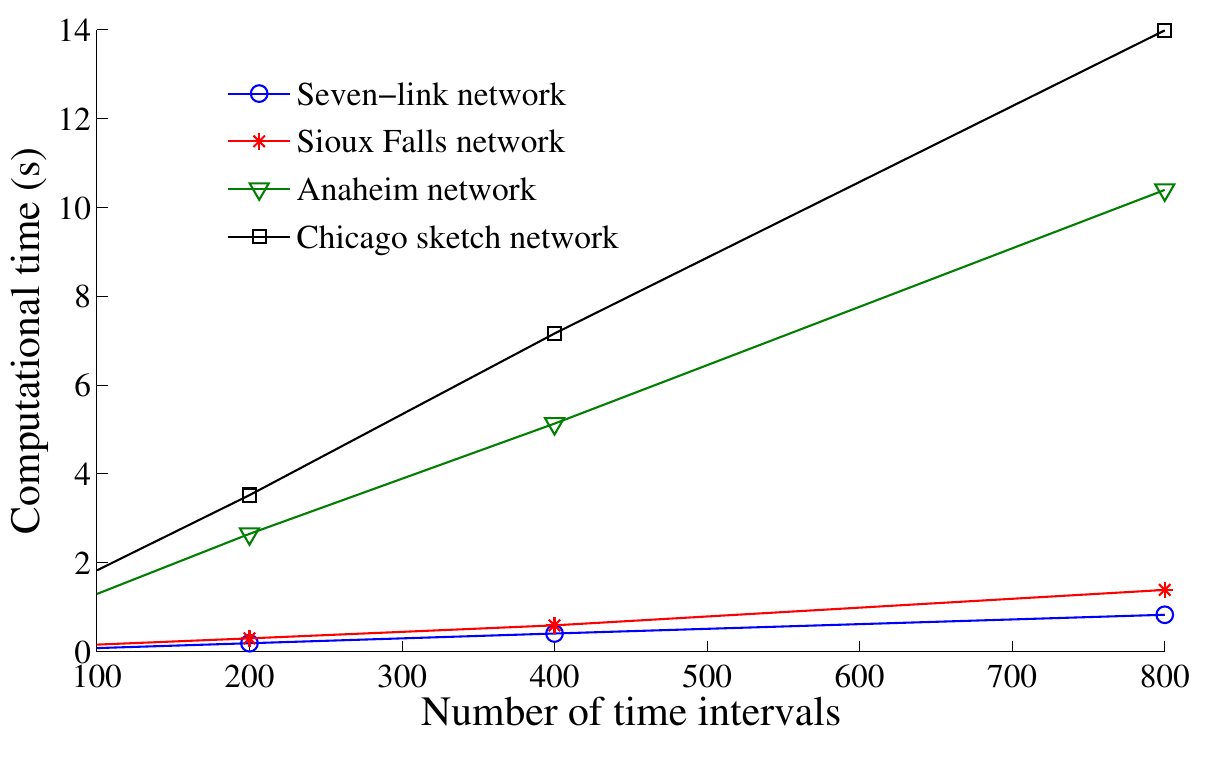}
\end{minipage}
\begin{minipage}[b]{.49\textwidth}
\centering
\includegraphics[width=\textwidth]{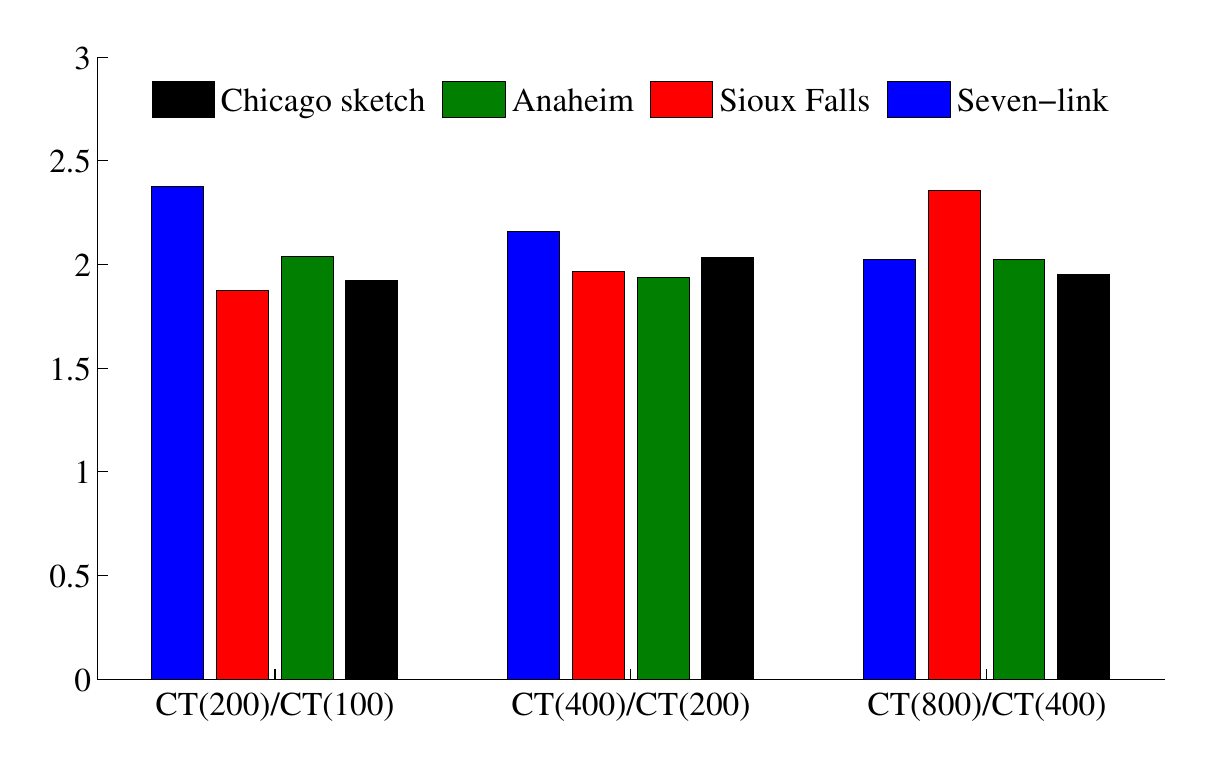}
\end{minipage}
\caption{Results of the computations results on the four test networks.}
\label{figCPUtime}
\end{figure}

The linear growth in the computational times are further visualized in Figure \ref{figCPUtime} on the left. On the right hand side of Figure \ref{figCPUtime}, we show the ratios $CT(200)/CT(100)$, $CT(400)/CT(200)$ and $CT(800)/CT(400)$, where $CT(N)$ denotes the computational time with $N$ time steps. It is  expected theoretically that as $N$ doubles, so should the computational time. This is confirmed in Figure \ref{figCPUtime}, where all the ratios are around 2.

\section{Concluding remarks}\label{secconclusion}

This paper presents a continuous-time version of the network-based LWR model. We show that the variational theory, when applied to  a triangular fundamental diagram, leads to a DAE representation of the network dynamics. This DAE system is link-based as it is concerned  with only variables associated with links in the network. The DAE system supports efficient computation as we have demonstrated using the numerical examples on several small and large networks.  We also show that the DAE system, when discretized by a proper forward scheme, leads to the link transmission model.

The existence and well-posedness of the network solution are also investigated in connection with the two specific Riemann Solvers presented in this paper. In deriving the existence result, we employ the general framework put forward by \cite{GP2009} and illustrate how the sufficient conditions for solution existence can be checked against a particular Riemann Solver. The well-posedness of solutions is shown using the notion of generalized tangent vector.

In general, solution existence and well-posedness are not guaranteed for arbitrary Riemann Solvers and/or junction models. An important aspect of future research is to establish these qualitative results for more junction types and Riemann Solvers. In addition, the well-posedness result presented in this paper relies on the important assumption that the vehicle turning ratios are fixed constants, which is not longer true if one considers a network with given origin-destination pairs and paths choices of drivers. In the latter case, which is often referred to as the dynamics network loading (DNL) model in the literature of dynamic traffic assignment, the multiplication factors for the norms of tangent vectors are no longer bounded. As a result, the well-posedness may fail. Such a counter example, as well as sufficient conditions that guarantee the well-posedness of solutions of the DNL model, are investigated in another paper \citep{LWRcont}.

\section{Acknowledgements}

This work is jointly supported by the National Natural Science Foundation of China (71271183) and a grant from the Research Grants Council of the Hong Kong Special Administrative Region, China (HKU 17207214E). The authors are grateful to the three reviewers for their constructive comments.

 \appendix

\section{Proof of technical results}

\subsection{Proof of Lemma \ref{existencediverge}}\label{secapplemmaexistencediverge}

\begin{proof}
First, we notice that the Riemann Solvers  for both the diverge and merge junctions are expressed in terms of link demand and supply, and some exogenous parameters such as turning ratio and link capacity. Furthermore, by \eqref{demanddef}-\eqref{supplydef}, the demand and supply both depend on bad datum only. Thus Property (P1) holds for both junction models.

We now show property (P2) holds for this diverge junction. We first focus on an interacting wave from link $I_1$.  It can be seen from (\ref{fifodiv}) that the change in the flow through the junction is proportional to the flow jump of this interacting wave. More precisely, indicating by $\Delta$ the change due to a wave interaction, we can write, for a wave interacting from road $I_1$, that
\[
\Delta q_{in,2}~=~\alpha_{1,2} \Delta q_{out,1}, \qquad
\Delta q_{in,3}~=~\alpha_{1,3}\Delta q_{out,1}
\]
where $\Delta q_{out,1} = f_1(\rho^+_1)-f_1(\rho^-_1)$. Moreover, the flow through the junction is given by:
\begin{equation}\label{gammaest}
\left|\Gamma(t+)-\Gamma(t-)\right|~=~|\Delta q_{out,1}|
\end{equation}
The change in the flow total variation, according to \eqref{TVdef}, is estimated as:
\begin{align}
|\Delta TV(f)| ~=~&\Big|  |\Delta q_{in, 2}|+|\Delta q_{in, 3}|+|f_1(\rho_1)-f_1(\rho_1^+)|-|f_1(\rho_1)-f_1(\rho_1^-)| \Big|  \nonumber
\\
~\leq~&  |\Delta q_{in, 2}|+|\Delta q_{in, 3}|+ \left|  f_1(\rho_1^+)-f_1(\rho_1^-) \right|  \nonumber
\\
~=~& \alpha_{1,2} |\Delta q_{out, 1}|+ \alpha_{1,3}|\Delta q_{out,1}|+|\Delta q_{out,1}| \nonumber
\\
\label{deltaTV}
~=~& (1+\alpha_{1,2}+\alpha_{1,3})|\Delta q_{out,1}|
\end{align}

\noindent Next, we claim that 
\begin{equation}\label{claimeqn}
|\Delta q_{out,1}|~=~|f_1(\rho_1^+)-f_1(\rho_1^-)|\leq|f_1(\rho_1)-f_1(\rho_1^-)|
\end{equation}
 Indeed, recall from \eqref{fifodiv} that 
\begin{equation}\label{fifodiv1}
q_{out,1}~=~\min\left\{D_1,\,{S_2\over \alpha_{1,2}},\,{S_3\over \alpha_{1,3}}\right\}
\end{equation}
Furthermore, recall the demand and supply viewed as functions of the Riemann data:
$$
D_i(\rho_i)~=~\begin{cases}
f_i(\rho_i) \quad & \rho_i\in[0,\,\sigma_i]
\\
C_i\quad  &\rho_i\in(\sigma_i,\,\rho_i^{jam}]
\end{cases}
\qquad
S_j(\rho_j)~=~\begin{cases}
C_j\quad & \rho_j\in[0,\,\sigma_j]
\\
f_j(\rho_j)\quad  &\rho_i\in(\sigma_j,\,\rho_j^{jam}]
\end{cases}
$$
where $\sigma_i$, $\rho_i^{jam}$, and $C_i$ denote the critical density, the jam density, and the flow capacity of link $I_i$, respectively. We consider several cases. \\

\noindent 1.  If $D_1(\rho_1)\leq \min\left\{{S_2\over \alpha_{1,2}},\,{S_3\over \alpha_{1,3}}\right\}$. Then it holds that $\rho_1^+=\rho_1$. Thus
$$
|\Delta q_{out,1}|~=~|f_1(\rho_1^+)-f_1(\rho_1^-)|~=~|f_1(\rho_1)- f_1(\rho_1^-)|
$$

\noindent 2. If $D_1(\rho_1) > \min\left\{{S_2\over \alpha_{1,2}},\,{S_3\over \alpha_{1,3}}\right\}\geq D_1(\rho_1^-)$. Then after the interaction, the exit flow satisfies $q_{out,1}=f_1(\rho_1^+)<D_1(\rho_1)= f_1(\rho_1)$. Then 
$$
0~\leq~\Delta q_{out,1}~=~f_1(\rho_1^+)-f_1(\rho_1^-)~<~f_1(\rho_1)-f_1(\rho_1^-)
$$

\noindent 3. If both $D_1(\rho_1)$ and  $D_1(\rho_1^-)$ are greater than $\min\left\{{S_2\over \alpha_{1,2}},\,{S_3\over \alpha_{1,3}}\right\}$. Then we have that 
$$
f_1(\rho_1^-)~=~f_1(\rho_1^+)=\min\left\{{S_2\over \alpha_{1,2}},\,{S_3\over \alpha_{1,3}}\right\}~\Longrightarrow ~|\Delta q_{out,1}|~=~0~\leq~|f_1(\rho_1)- f_1(\rho_1^-)|
$$
In any case, the claimed inequality holds. Thus, the following holds according to \eqref{gammaest}, \eqref{deltaTV} and \eqref{claimeqn}.
$$
|\Delta TV(f)|\leq(1+\alpha_{1,2}+\alpha_{1,3})|\Delta q_{out,1}|=(1+\alpha_{1,2}+\alpha_{1,3})\min\left\{\left|\Gamma(t+)-\Gamma(t-)\right|,\, |f_1(\rho_1)-f_1(\rho_1^-)|\right\}
$$
This verifies (P2) for an interacting wave from $I_1$.

We now consider an interacting wave from $I_2$, and begin with the following estimations:
\begin{equation}\label{lemmadiveqn1}
\Delta q_{out,1}~=~\frac{1}{\alpha_{1,2}}\Delta q_{in,2},\qquad
\Delta q_{in,3}~=~\frac{\alpha_{1,3}}{\alpha_{1,2}}\Delta q_{in,2}
\end{equation}
Similar to \eqref{deltaTV}, we have the following inequality:
\begin{equation}\label{lemmadiveqn2}
|\Delta TV(f)|~\leq~\left(1+{1\over \alpha_{1,2}}+{\alpha_{1,3}\over\alpha_{1,2}}\right)|\Delta q_{in,2}|
\end{equation}
In addition, one can derive the inequality 
\begin{equation}\label{claimeqn2}
|\Delta q_{in,2}|\leq |f_2(\rho_2)-f_2(\rho_2^-)|
\end{equation}
in the same way as before, the proof of which is straightforward but tedious and is omit from this paper. Finally, according to \eqref{lemmadiveqn1}, \eqref{lemmadiveqn2} and \eqref{claimeqn2},
\begin{align*}
|\Delta TV(f)|~\leq~&\left(1+{1\over \alpha_{1,2}}+{\alpha_{1,3}\over\alpha_{1,2}}\right)|\Delta q_{in,2}|
\\
~=~&\left(1+{1\over \alpha_{1,2}}+{\alpha_{1,3}\over\alpha_{1,2}}\right)\min\left\{|\Delta q_{in,2}|,\,|f_2(\rho_2)-f_2(\rho_2^-)|\right\}
\\
~\leq~&\left(1+{1\over \alpha_{1,2}}+{\alpha_{1,3}\over\alpha_{1,2}}\right) \min\left\{ |\Gamma(t+)-\Gamma(t-)|,\, |f_2(\rho_2)-f_2(\rho_2^-)|\right\}
\end{align*}
This establishes (P2) for link $I_2$. The case with link $I_3$ is similar.

For (P3), we notice from \eqref{fifodiv} and the definitions of demand and supply that a wave bringing a decrease in flow also decreases the demand (for incoming links) or the supply (for outgoing links), and thereby decreases the flow through the junction.  Therefore,  (P3) is granted. 

Finally, with all three properties verified, Theorem 5.1 of \cite{GP2009} asserts that a weak solution exists at the diverge junction. 
\end{proof}

\subsection{Proof of Lemma \ref{existencemerge}}\label{secappexistencemerge}

\begin{proof} 
\noindent Property (P1) was already discussed in the proof of Lemma \ref{existencediverge}. The rest of the proof is divided into several parts.\\

\noindent{\bf (Part 1.)} Let us verify property (P2). First, assume there is a wave interacting from road $I_4$ (with the case of $I_5$ being entirely similar). Using the same notations as in the previous lemma, we indicate by $(\rho_4,\,\rho_4^-)$ the interacting wave and denote by $D_4(\rho_4^-)$ and $f_4(\rho^-_4)$ the demand and exit flow of $I_4$, respectively, before the interaction. In addition, let $D_4(\rho_4)$ the demand at the time of interaction.
We now distinguish between two cases:
\begin{itemize}
\item[(a)] $ D_4(\rho^-_4)>f_4(\rho_4^-)$ (corresponding to the left part of Figure \ref{figmerge}).
\item[(b)] $ D_4(\rho^-_4)=f_4(\rho_4^-)$ (corresponding to the right part of Figure \ref{figmerge}).
\end{itemize}

\noindent{\bf (Part 1.a.)} Case a) implies that $\rho_4^-$ is a good datum; that is, $\rho_4^->\sigma_4$.  Thus $\rho_4$ must be a bad datum because of the positive wave speed and the Rankine-Hugoniot condition. In addition, we have that $D_4(\rho_4)=f_4(\rho_4)<f_4(\rho_4^-)< D_4(\rho_4^-)$. According to the Riemann Solver illustrated in the left part of Figure \ref{figmerge}, we must have $f_4(\rho_4^+)=D_4(\rho_4)$ and $\rho_4^+=\rho_4$ after the interaction. Accordingly, the change in the total variation of the flow 
\begin{align*}
\Delta TV(f)~=~&  |f_5(\rho_5^+) -f_5(\rho_5^-)|+|f_6(\rho_6^+)-f_6(\rho_6^-)| +|f_4(\rho_4)-f_4(\rho_4^+)|-|f_4(\rho_4)-f_4(\rho_4^-)| 
\\
~=~&  |f_5(\rho_5^+) -f_5(\rho_5^-)|+|f_6(\rho_6^+)-f_6(\rho_6^-)| -|f_4(\rho_4)-f_4(\rho_4^-)|
\end{align*}
We need to estimate changes in the flows on other links.  In particular, we consider two cases: \\
\noindent (a1) $D_4(\rho_4)> S_6(\rho_6^-)-D_5(\rho_5^-)$; and \\
\noindent (a2) $D_4(\rho_4)\leq S_6(\rho_6^-)- D_5(\rho_5^-)$.\\
 The left part of Figure \ref{figCase1234} provides an illustration of these two cases. In case (a1), we easily deduce that 
$$
\Delta TV(f)  ~=~ f_4(\rho_4^-) -D_4(\rho_4) + 0 - (f_4(\rho_4^-)-f_4(\rho_4))  ~=~0
$$
In case (a2), we have that 
\begin{align*}
\Delta TV(f)~=~& D_5(\rho_5^-)-(S_6(\rho_6^-)-f_4(\rho_4^-))+S_6(\rho_6^-)-(D_4(\rho_4)+D_5(\rho_5^-)) - (f_4(\rho_4^-)-f_4(\rho_4))     
\\
~=~&0
\end{align*}
We conclude that under case (a), $\Delta TV(f)=0$.\\

\begin{figure}[h!]
\centering
\includegraphics[width=\textwidth]{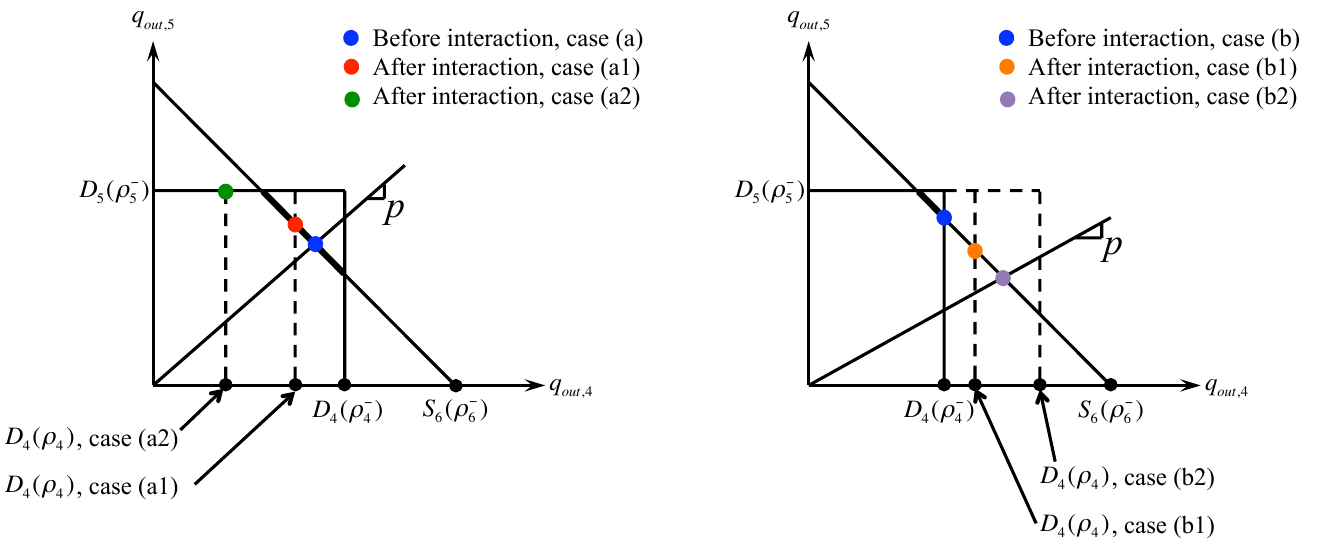}
\caption{Illustration of all the cases in the proof of Lemma \ref{existencemerge}.}
\label{figCase1234}
\end{figure}

\noindent{\bf (Part 1.b.)} We now turn to case (b) and deduce that $\rho_4^-$ is a bad datum. According to the direction of the wave, $\rho_4$ must also be a bad datum.  If $\rho_4<\rho_4^-$, then $D_4(\rho_4)=f_4(\rho_4)<f_4(\rho_4^-)$, which leads to the exact same situation as Case a) and we conclude that $\Delta TV(f)=0$. On the other hand, if $\rho_4^-<\rho_4<\sigma_4$ we deduce $D_4(\rho_4)>f_4(\rho_4^-)$ and $f_4(\rho_4)>f_4(\rho_4^-)$.  We again consider two cases:\\
\noindent (b1) $D_4(\rho_4)\leq S_6(\rho_6^-)/(1+p)$; and \\
\noindent (b2) $D_4(\rho_4)> S_6(\rho_6^-)/(1+p)$.\\
 See the right part of Figure \ref{figCase1234} for an illustration of these two cases. In both cases, we see that the inflow into link $I_6$ does not change after the interaction, thus
\begin{align*}
\Delta TV(f)~=~&  |f_5(\rho_5^+) -f_5(\rho_5^-)|+|f_6(\rho_6^+)-f_6(\rho_6^-)| +|f_4(\rho_4)-f_4(\rho_4^+)|-|f_4(\rho_4)-f_4(\rho_4^-)| 
\\
~=~&  |f_5(\rho_5^+) -f_5(\rho_5^-)|+ |f_4(\rho_4)-f_4(\rho_4^+)|-|f_4(\rho_4)-f_4(\rho_4^-)| 
\end{align*}
In case (b1), we have that $D_4(\rho_4)=f_4(\rho_4^+)> f_4(\rho_4)$ and thus $\rho_4^+=\rho_4$. Therefore,
\begin{align*}
\Delta TV(f)~=~&   D_4(\rho_4)-D_4(\rho_4^-) + |f_4(\rho_4)-f_4(\rho_4^+)|-(f_4(\rho_4)-f_4(\rho_4^-)) 
\\
~=~&  f_4(\rho_4^+)-f_4(\rho_4^-)+f_4(\rho_4^+)-f_4(\rho_4)-f_4(\rho_4)+f_4(\rho_4^-)~=~0
\end{align*}
In case (b2), we have that $f_4(\rho_4^+)<D_4(\rho_4)=f_4(\rho_4)$ and $\rho_4^+$ is a good datum. Consequently,
\begin{align*}
\Delta TV(f)~=~&  {S_6(\rho_6^-)\over 1+p}-f_4(\rho_4^-) + f_4(\rho_4)-f_4(\rho_4^+)-(f_4(\rho_4)-f_4(\rho_4^-))     
\\
~=~&{S_6(\rho_6^-)\over 1+p}  -f_4(\rho_4^+)~=~0
\end{align*}
We conclude that case (b) yields $\Delta TV(f)=0$ as well. This easily verifies  (P2) in the case of  wave coming from $I_4$ (or $I_5$).\\

\noindent{\bf (Part 3.)} We consider two cases\\
(i) $\rho_6^-$ is good datum; and \\
(ii) $\rho_6^-$ is bad datum. \\
In case (i), a wave can reach the junction from $I_6$ only if it has a negative speed, and hence, $\rho_6>\sigma_6$ and $f_6(\rho_6)<f_6(\rho_6^-)$. Before the interaction time $t$, we have
$$
\Gamma(t-)~=~f_6(\rho_6^-)~>~f_6(\rho_6)
$$
Therefore, after the interaction time the new solution given by the Riemann Solver is attained at $S_6(\rho_6)=f_6(\rho_6)$. Thus $\Gamma(t+)=f_6(\rho_6)$,  and $\rho_6^+=\rho_6$. We then have
$$
\Gamma(t+)-\Gamma(t-)~=~f_6(\rho_6)-f_6(\rho_6^-)
$$
By definition, the change in the total variation of flow is
$$
\Delta TV(f)~=~ |\Delta q_{out,4}| +|\Delta q_{out,5}| +|f_6(\rho_6)-f_6(\rho_6^+)| -|f_6(\rho_6)-f_6(\rho_6^-)|  
$$
Since the inflow of $I_6$ equals the sum of outflows of $I_4$ and $I_5$, the decrease in the junction flow brings decrease in the outflows of both $I_4$ and $I_5$, that is,
$$
|\Delta q_{out,4}|+|\Delta q_{out,5}|~=~|\Gamma(t-)-\Gamma(t+)|~=~f_6(\rho_6^-)-f_6(\rho_6)
$$ 
We thus easily verify  that $\Delta TV(f)=0$ in case (i). \\

\noindent In case (ii), $\rho_6^- > \sigma_6$ and hence $\rho_6\geq\sigma_6$. If $\rho_6>\rho_6^-$, then $f_6(\rho_6)<f_6(\rho_6^-)$ and we have exactly the same estimation as in case (i). Now assume that $\rho_6<\rho_6^-$, and thus $f_6(\rho_6)>f_6(\rho_6^-)=S_6(\rho_6^-)=\Gamma(t-)$. Consider 
two further cases: \\
(iia) $f_6(\rho_6)\leq D_4(\rho_4^-)+D_5(\rho_5^-)$; and \\
(iib) $f_6(\rho_6)> D_4(\rho_4^-)+D_5(\rho_5^-)$. \\
For (iia), we have that $\rho_6^+=\rho_6$, and the increase in the junction flow after the interaction is reflected by the simultaneous increase in exit flows from $I_4$ and $I_5$. In other words,
$$
|\Delta q_4|+|\Delta q_5|~=~f_6(\rho_6)-f_6(\rho_6^-)
$$
Similar to case (i), we must have $\Delta TV(f)=0$. \\
For (iib), we have 
\begin{equation}\label{iibeqn1}
\begin{array}{c}
f_6(\rho_6^+)~=~\Gamma(t+)~=~D_4(\rho_4^-)+D_5(\rho_5^-)~<~f_6(\rho_6)
\\
 f_6(\rho_6^-)~=~\Gamma(t-)~\leq~ D_4(\rho_4^-)+D_5(\rho_5^-)~<~f_6(\rho_6)
 \\
 f_6(\rho_6^-)~=~\Gamma(t-)~\leq~ D_4(\rho_4^-)+D_5(\rho_5^-)~=~f_6(\rho_6^+)
\end{array}
\end{equation}
In addition,
$$
\Delta\Gamma~=~\Gamma(t+)-\Gamma(t-)~=~f(\rho_6^+)-f(\rho_6^-)
$$
Finally, according to \eqref{iibeqn1} we have
\begin{align*}
\Delta TV(f)~=~&|f_4(\rho_4^+)-f_4(\rho_4^-)|+|f_5(\rho_5^+)-f_5(\rho_5^-)|+|f_6(\rho_6^+)-f_6(\rho_6)| - |f_6(\rho_6)-f_6(\rho_6^-)|
\\
~=~&|f_6(\rho_6^+)-f_6(\rho_6^-)|+|f_6(\rho_6^+)-f_6(\rho_6)| -|f_6(\rho_6)-f_6(\rho_6^-)|
\\
~=~&f_6(\rho_6^+)-f_6(\rho_6^-)+f_6(\rho_6)-f_6(\rho_6^+) -f_6(\rho_6)+f_6(\rho_6^-)
\\
~=~&0
\end{align*}

\noindent {(\bf Part 4.)} We verify {\bf (P3)} here. Notice that, by the definition of demand and supply, a wave bringing a decrease in the flow also reduces the corresponding demand (for incoming links) and supply (for outgoing links) of the relevant link, and thereby decreases the flow through the junction after the interaction, according to Remark \ref{remarkmergemono}.

Finally, with all three properties verified, Theorem 5.1 of \cite{GP2009} asserts that a weak solution exists at the merge junction. 
\end{proof}


\end{document}